\documentclass[a4paper,10pt]{article}

\usepackage[table,dvipsnames]{xcolor}
\usepackage[T1]{fontenc}
\usepackage{amsmath,amsfonts,amsthm,stmaryrd}
\usepackage{amssymb}
\usepackage{a4wide}
\usepackage{authblk}
\usepackage[latin1]{inputenc}
\usepackage{mathtools}
\usepackage{mathabx}
\usepackage{nicefrac}
\usepackage{numprint} 
\usepackage{paralist}
\usepackage{hyperref}
\usepackage{pgfplots,pgfplotstable}
\usepackage[font={footnotesize}]{caption}
\usepackage[font={footnotesize}]{subcaption}
\usepackage[bold]{hhtensor}
\usepackage{graphicx}
\usepackage{booktabs}
\usepackage{cite}
\usepackage{comment}
\usepackage{enumerate}
\usepackage{tikz}
\usetikzlibrary{positioning,calc,shapes.geometric,arrows}
\usepackage{stackengine}
\stackMath

\newcommand{\com}[1]{}

\graphicspath{{./figures/pdf/}}

\usepackage{parskip}

\makeatletter
\def\thm@space@setup{%
  \thm@preskip=\parskip \thm@postskip=0pt
}
\makeatother

\newcommand{\email}[1]{\href{mailto:#1}{#1}}

\pgfqkeys{/pgfplots}{
  cycle list name = black white,
}

\newcommand{\GRAD}{\matr\nabla}
\newcommand{\vGRAD}{\vec\nabla}
\newcommand{\GRADs}[1]{\GRAD_{\hskip-0.5ex {\mathrm s}}\hskip0.2ex#1}

\newcommand{\DIV}{\vec\nabla{\cdot}}

\newcommand{\eqbydef}{\mathrel{\mathop:}=}
\newcommand{\fedybqe}{=\mathrel{\mathop:}}

\newcommand{\restrto}[2]{#1{}_{|#2}}

\newcommand{\Id}[1][d]{\matr{I}_{#1}}

\newcommand{\del}{\partial}

\newcommand{\Real}{\mathbb{R}}
\newcommand{\Natural}{\mathbb{N}}
\newcommand{\IR}{\Real}
\newcommand{\IN}{\Natural}

\newcommand{\IP}[1][]{\mathbb{P}^{#1}}
\newcommand{\tvIP}[1][]{\tilde{\vec{\mathbb P}}^{#1}}

\newcommand{\Ltd}[1][\Omega]{[L^2(#1)]^d}
\newcommand{\Ltdd}[1][\Omega]{[L^2(#1)]^{d\times d}}

\newcommand{\Lt}[1][\Omega]{L^2(#1)}

\newcommand{\Hozd}[1][\Omega]{[H^1_0(#1)]^d}
\newcommand{\Hod}[1][\Omega]{[H^1(#1)]^d}

\newcommand{\Hdiv}[1][\Omega]{\mathbb{H}(\mathrm{div},#1)}
\newcommand{\Hdivs}[1][\Omega]{\mathbb{H}_{\mathrm s}({\mathrm{div}},#1)}

\newcommand{\oma}{\omega_a}
\newcommand{\Th}[1][h]{\mathcal{T}_{#1}}
\newcommand{\Ta}{\mathcal{T}_a}
\newcommand{\Vh}[1][h]{\mathcal{V}_{#1}}

\newcommand{\Vint}{\mathcal{V}_h^{\mathrm{int}}}
\newcommand{\Vext}{\mathcal{V}_h^{\mathrm{ext}}}

\newcommand{\VT}{\mathcal{V}_T}
\newcommand{\Fh}[1][h]{\mathcal{F}_{#1}}

\newcommand{\jump}[1]{\llbracket#1\rrbracket}

\newcommand{\nor}[1]{\lVert #1 \rVert}
\newcommand{\norm}[2][]{\lvert#2\rvert_{#1}}

\newcommand{\term}{\mathfrak{T}}

\DeclareMathOperator{\optr}{tr}
\DeclareMathOperator{\opdev}{dev}

\newcommand{\tenf}[1]{\stackunder[0pt]{\stackunder[1pt]{#1}{\scriptscriptstyle\sim}}{\scriptscriptstyle \sim}}

\newcommand{\ms}[1][]{\matr{\sigma}_{\hskip-0.2ex #1}}
\newcommand{\msh}{\ms[h]}
\newcommand{\mshdisc}{\ms[h,{\rm disc}]^k}
\newcommand{\mshlin}{\ms[h,{\rm lin}]^k}
\newcommand{\mt}[1][]{\matr{\tau}_{\hskip-0.3ex #1}}
\newcommand{\mth}{\mt[h]}
\newcommand{\mS}[1][]{\matr{\Sigma}_{#1}}
\newcommand{\tmS}{\tilde{\matr\Sigma}_{h}^a}
\newcommand{\tmST}{\tilde{\matr\Sigma}_{T}}

\newcommand{\vn}[1][]{\vec{n}_{#1}}
\newcommand{\vx}[1][]{\vec{x}}
\newcommand{\vu}[1][]{\vec{u}_{#1}}

\newcommand{\vv}[1][]{\vec{v}_{#1}}
\newcommand{\vV}[1][]{\vec{V}_{\hskip-0.2ex #1}}
\newcommand{\vw}[1][]{\vec{w}_{#1}}

\newcommand{\vz}[1][]{\vec{z}_{#1}}
\newcommand{\vy}[1][]{\vec{y}_{#1}}
\newcommand{\vM}[1][]{\vec{M}_{\hskip-0.2ex #1}}

\newcommand{\vf}[1][]{\vec{f}_{#1}}

\newcommand{\uh}{\vec u_{h}}
\newcommand{\vmh}{\vec{m}_{h}}
\newcommand{\uhk}[1][k]{\vec u_{h}^{#1}}
\newcommand{\vh}{\vec v_{h}}

\newcommand{\sha}{\matr\sigma_{\hskip-0.1ex h}^a}
\newcommand{\msu}[1][\vu]{\matr{\sigma}(\GRADs{#1})}
\newcommand{\omsu}[1][\uhk]{\overline{\matr{\sigma}}(\GRADs{#1})}

\newcommand{\msuh}{\msu[\uh]}
\newcommand{\mslinea}[1][\uhk]{\ms^{k-1}(\GRADs{#1})}

\newtheorem{theorem}{Theorem}
\newtheorem{assumption}[theorem]{Assumption}
\newtheorem{remark}[theorem]{Remark}

\newtheorem{proposition}[theorem]{Proposition}
\newtheorem{lemma}[theorem]{Lemma}

\newtheorem{algo}[theorem]{Algorithm}
\newtheorem{const}[theorem]{Construction}
\newtheorem{example}[theorem]{Example}
\numberwithin{equation}{section}
\numberwithin{theorem}{section}

\title{Equilibrated stress tensor reconstruction and a posteriori error estimation for nonlinear elasticity}
\author[1]{Michele Botti\footnote{\email{michele.botti01@universitadipavia.it}}}
\author[1,2]{Rita Riedlbeck\footnote{\email{rita.riedlbeck@umontpellier.fr}}}
\affil[1]{
  University of Montpellier, Institut Montp\'ellierain Alexander Grothendieck, 34095 Montpellier, France
}
\affil[2]{
  EDF R\&D, IMSIA, 91120 Palaiseau, France
}

\begin{document}
\maketitle

\begin{abstract}
We consider hyperelastic problems and their numerical solution using a conforming finite element discretization and iterative linearization algorithms. 
For these problems, we present equilibrated, weakly symmetric, $H(\rm{div)}$-conforming stress tensor reconstructions, obtained from local problems on patches around vertices using the Arnold--Falk--Winther finite element spaces.
We distinguish two stress reconstructions, one for the discrete stress and one representing the linearization error.
The reconstructions are independent of the mechanical behavior law.
Based on these stress tensor reconstructions, we derive an a posteriori error estimate distinguishing the discretization, linearization, and quadrature error estimates, and propose an adaptive algorithm balancing these different error sources.
We prove the efficiency of the estimate, and confirm it on a numerical test with analytical solution for the linear elasticity problem. 
We then apply the adaptive algorithm to a more application-oriented test, considering the Hencky--Mises and an isotropic damage models.
\end{abstract}

\section{Introduction}
\label{sec:introduction}

In this work we develop equilibrated $H(\rm{div})$-conforming stress tensor reconstructions for a class of (linear and) nonlinear elasticity problems in the small deformation regime.
Based on these reconstructions, we can derive an a posteriori error estimate distinguishing the discretization and linearization errors for conforming discretizations of the problem.

Let $\Omega\in\IR^d$, $d\in\{2,3\}$, be a bounded, simply connected polyhedron, which is occupied by a body subjected to a volumetric force field $\vf\in\Ltd$.
For the sake of simplicity, we assume that the body is fixed on its boundary $\partial\Omega$.
The nonlinear elasticity problem consists in finding a vector-valued displacement field $\vu : \Omega\to\IR^d$ solving
\begin{subequations}
  \label{eq:ne.strong}
  \begin{alignat}{4}
    \label{eq:ne.mech}
    -\DIV\msu &= \vf &\qquad &\text{in}\;\Omega,
    \\
    \label{eq:ne.dirichlet}
    \vu &= \vec{0} &\qquad &\text{on}\;\partial\Omega,
  \end{alignat}
\end{subequations} 
where $\GRADs{\vu}=\frac12 ((\GRAD\vu)^T+\GRAD\vu)$ denotes the symmetric gradient and expresses the strain tensor associated to $\vu$.
The stress-strain law $\ms:\Omega\times\Real^{d\times d}_{\mathrm{sym}}\to
\Real^{d\times d}_{\mathrm{sym}}$ is assumed to satisfy regularity requirements inspired by \cite{BotDPiSoc17,DPiDro15,DPiDro16}.
Problem \eqref{eq:ne.strong} describes the mechanical behavior of soft materials \cite{Tre75} and metal alloys \cite{PitZan02}. 
Examples of stress-strain relations of common use in the engineering practice are given in Section~\ref{sec:setting}. 
In these applications, the solution is often approximated using $H^1$-conforming finite elements. 
For nonlinear mechanical behavior laws, the resulting discrete nonlinear equation can then be solved using an iterative linearization algorithm yielding at each iteration a linear algebraic system to be solved, until the residual of the nonlinear equation lies under a predefined threshold.

In this paper we develop an a posteriori error estimate allowing to distinguish between the error stemming from the linearization of the problem and the one due to its discretization, as proposed in \cite{ErnVoh13} for nonlinear diffusion problems. Thanks to this distinction we can, at each iteration, compare these two error contributions and stop the linearization algorithm once its contribution is negligible compared to the discretization error. 

The a posteriori error estimate is based on equilibrated stress reconstructions.
It is well known that, in contrast to the analytical solution, the discrete stress tensor resulting from the conforming finite element method does not have continuous normal components across mesh interfaces, and that its divergence is not locally in equilibrium with the source term $\vf$ on mesh elements.
In this paper we consider the stress tensor reconstruction proposed in \cite{fvca8} for linear elasticity to restore these two properties.
This reconstruction uses the Arnold--Falk--Winther mixed finite element spaces \cite{ArnFalWin07}, leading to weakly symmetric tensors . 
In \cite{fvca8} this reconstruction is compared to a similar reconstruction introduced in \cite{RieDPiErn17} using the Arnold--Winther finite element spaces \cite{ArnWin02}, yielding a symmetric tensor, and very good agreement was observed while saving substantial computational effort.
In Section \ref{sec:reconst} we apply this reconstruction to the nonlinear case by constructing two stress tensors: one playing the role of the discrete stress and one expressing the linearization error.
They are obtained by summing up the solutions of constrained minimization problems on cell patches around each mesh vertex, so that they are $H({\rm div})$-conforming and the sum of the two reconstructions verifies locally the mechanical equilibrium \eqref{eq:ne.mech}. 
The patch-wise equilibration technique was introduced in \cite{DesMet99,BraeSchoe08} for the Poisson problem using the Raviart--Thomas finite element spaces. In \cite{DoerMel13} it is extended to linear elasticity without any symmetry constraint by using linewise Raviart--Thomas reconstructions. 
Elementwise reconstructions from local Neumann problems requiring some pre-computations to determine the normal fluxes to obtain an equilibrated stress tensor can be found in \cite{AinRan10,ChaLadPle12,LadPelRou91,OhnSteiWal01}, whereas in \cite{NicWitWoh08} the direct prescription of the degrees of freedom in the Arnold--Winther finite element space is considered.

Based on the equilibrated stress reconstructions, we develop the a posteriori error estimate in Section~\ref{sec:a_posteriori} and  prove that this error estimate is efficient, meaning that, up to a generic constant, it is also a local lower bound for the error.
The idea goes back to \cite{PraSyn47} and was advanced amongst others by \cite{Lad75,LadLeg83,AinOde00,Rep08} for the upper bound.
Local lower error bounds are derived in \cite{DesMet99,ErnVoh15,LucWoh04,BraeSchoe08,ErnVoh10}.
Using equilibrated fluxes for a posteriori error estimation offers several advantages. 
The first one is, as mentioned above, the possible distinction and comparison of error components by expressing them in terms of fluxes. 
Secondly, the error upper bound is obtained with fully computable constants. In our case these constants depend only on the parameters of the stress-strain relation. 
Thirdly, since the estimate is based on the discrete stress (and not the displacement), it does not depend on the mechanical behavior law (except for the constant in the upper bound).
Therefore, its implementation is independent and directly applicable to these laws, which makes the method convenient for FEM softwares in solid mechanics, which often provide a large choice of behavior laws. 
In addition, equilibrated error estimates were proven to be polynomial-degree robust for several linear problems in 2D, as the Poisson problem in \cite{BraePilSchoe09,ErnVoh15}, linear elasticity in \cite{DoerMel13} and the related Stokes problem in \cite{CerHecTan17} and recently in 3D in \cite{ErnVoh17}. 

This paper is organized as follows. 
In Section~\ref{sec:setting} we first formulate the assumptions on the stress-strain function $\ms$ and provide three examples of models used in the engineering practice. We then introduce the weak and the discrete formulations of problem \eqref{eq:ne.strong} and its linearization, along with some useful notation. 
In Section \ref{sec:reconst} we present the equilibrated stress tensor reconstructions, first assuming that we solve the nonlinear discrete problem exactly, and then, based on this first reconstruction, distinguish its discrete and its linearization error part at each iteration of a linearization solver.
In Section \ref{sec:a_posteriori} we derive the a posteriori error estimate, again first assuming the exact solution of the discrete problem and then distinguishing the different error components. We then propose an algorithm equilibrating the error sources using adaptive stopping criteria for the linearization and adaptive remeshing. We finally show the efficiency of the error estimate.
In Section \ref{sec:tests} we evaluate the performance of the estimates for the three behavior laws given as examples on numerical test cases.

\section{Setting} \label{sec:setting}

In this section we will give three examples of hyperelastic behavior laws, before writing the weak and the considered discrete formulation of problem \eqref{eq:ne.strong}.

\subsection{Continuous setting}

\begin{assumption}[Stress-strain relation]
  \label{ass:hypo}
We assume that the symmetric stress tensor $\ms:\Real^{d\times d}_{\mathrm{sym}}\to
\Real^{d\times d}_{\mathrm{sym}}$ is continuous on $\Real^{d\times d}_{\mathrm{sym}}$ and that $\ms(\matr{0})=\matr{0}$. Moreover, we assume that there exist real numbers $C_{\rm gro},C_{\rm mon}\in(0,+\infty)$ such that, for all 
$\mt,\matr{\eta}\in\Real^{d\times d}_{\mathrm{sym}}$,
\begin{subequations}
  \begin{alignat}{2} 
    \label{eq:hypo.gr}
    \norm[d\times d]{\ms(\mt)}&\le C_{\rm gro}\norm[d\times d]{\mt},
    &\quad &\text{(growth)}
    \\ 
    \label{eq:hypo.smon}
    \left(\ms(\mt)-\ms(\matr\eta)\right):\left(\mt-\matr{\eta}\right)&\ge
    C_{\rm mon}^2\norm[d\times d]{\mt-\matr\eta}^2,
    &\quad &\text{(strong monotonicity)}
  \end{alignat}
\end{subequations}
where $\mt:\matr{\eta}\eqbydef\optr(\mt^T\matr{\eta})$ 
with $\optr(\mt)\eqbydef\sum_{i=1}^d \tau_{ii}$, and $|\mt|^2_{d\times d}=\mt:\mt$. 
\end{assumption}

We next discuss a number of meaningful stress-strain relations for hyperelastic materials that satisfy the above assumptions. Hyperelasticity is a type of constitutive model for ideally elastic materials in which the 
stress is determined by the current state of deformation by deriving a stored energy density function 
$\Psi:\Real^{d\times d}_{\rm{sym}}\to\Real$, namely 
$$
\ms(\mt) \eqbydef \frac{\partial\Psi(\mt)}{\partial\mt}.
$$
\begin{example}[Linear elasticity]\label{ex:lin.cauchy}
  The stored energy density function leading to the linear elasticity model is
  \begin{equation}
    \label{eq:stored_energy_lin}
    \Psi_{\rm{lin}}(\mt) \eqbydef \frac\lambda2 \optr (\mt)^2 + \mu\optr(\mt^2),
  \end{equation} 
  where $\mu>0$ and $\lambda\ge0$ are the Lam\'e parameters.
  Deriving~\eqref{eq:stored_energy_lin} yields the usual Cauchy stress tensor
  \begin{equation}
    \label{eq:LinCauchy}
    \ms(\mt)=\lambda\optr(\mt)\Id+2\mu\mt.
  \end{equation}
  Being linear, the previous stress-strain relation clearly satisfies Assumption~\ref{ass:hypo}. 
\end{example}

\begin{example}[Hencky--Mises model]\label{ex:hencky}
  The nonlinear Hencky--Mises model of~\cite{Nec86,GatSte02} corresponds to the
  stored energy density function
  \begin{equation}
    \label{eq:stored_energy_HM}
    \Psi_{\rm{hm}}(\mt) \eqbydef \frac\alpha2 \optr (\mt)^2 + \Phi(\opdev(\mt)),
  \end{equation}
  where $\opdev:\Real^{d\times d}_{\rm{sym}}\to\Real$ defined by $\opdev(\mt)
  =\optr (\mt^2)-\frac1d\optr (\mt)^2$ is the deviatoric operator.
  Here, $\alpha\in(0,+\infty)$ and $\Phi:[0,+\infty)\to\Real$ is a function of class $C^2$ satisfying, for some positive 
  constants $C_1$, $C_2$, and $C_3$,
  \begin{equation}
    \label{eq:condition_phi} 
    C_1\le\Phi'(\rho)<\alpha,\quad |\rho\Phi''(\rho)|\le C_2\quad\text{and}\quad\Phi'(\rho)+2\rho\Phi''(\rho)\ge C_3
    \;\quad\forall \rho\in[0,+\infty). 
  \end{equation}
  We observe that taking $\alpha=\lambda+\frac2d\mu$ and $\Phi(\rho)=\mu\rho$ in~\eqref{eq:stored_energy_HM} leads to the 
  linear case~\eqref{eq:stored_energy_lin}. Deriving the energy density function~\eqref{eq:stored_energy_HM} yields
    \begin{equation}
      \label{eq:HenckyMises}
      \ms(\mt)=\tilde{\lambda}(\opdev(\mt))\optr (\mt)\Id+2\tilde{\mu}(\opdev(\mt))\mt,
    \end{equation}  
  with nonlinear Lam\'e functions $\tilde\mu(\rho)\eqbydef\Phi'(\rho)$ 
  and $\tilde\lambda(\rho)\eqbydef\alpha-\Phi'(\rho)$. Under conditions~\eqref{eq:condition_phi} it can be proven that the 
  previous stress-strain relation satisfies Assumption~\ref{ass:hypo}. 
\end{example}

In the previous example the nonlinearity of the model only depends on the deviatoric part of the strain. In the following model it depends on the term $\mt:\tenf C\mt$.

\begin{example}[An isotropic reversible damage model]\label{ex:damage}
  The isotropic reversible damage model of~\cite{CerChiCod10} can also be interpreted in the framework of hyperelasticity by setting up the energy density function as
  \begin{equation}
    \label{eq:stored_energy_DAM}
    \Psi_{\rm{dam}}(\mt) \eqbydef \frac{(1-D(\mt))}{2}\mt:\tenf C\mt + \Phi(D(\mt)),
  \end{equation}
  where $D:\Real^{d\times d}_{\rm{sym}}\to[0,1]$ is the scalar damage function and $\tenf C$ is a fourth-order 
  symmetric and uniformly elliptic tensor, namely, for some positive constants $C_*$ and $C^*$, it holds
  \begin{equation}
    \label{eq:unif.elliptic}
    C_*\norm[d\times d]{\mt}^2\le \tenf C\mt:\mt\le C^*\norm[d\times d]{\mt}^2, 
    \quad\forall\mt\in\Real^{d\times d}.
  \end{equation} 
  The function $\Phi:[0,1]\to\Real$ defines the relation between $\mt$ and $D$ by $\frac{\partial{\phi}}{\partial{D}}=
  \frac12\mt:\tenf C\mt$. The resulting stress-strain relation reads
  \begin{equation}
    \label{eq:Damage}
    \ms(\mt)=(1-D(\mt))\tenf C\com{(\cdot)}\mt.
  \end{equation}
  
  If there exists a continuous function $f:\lbrack 0,+\infty)\to[a,b]$ for some $0<a\le b\le1$, such that 
  $s\in\lbrack 0,+\infty)\to sf(s)$ is strictly increasing and, for all $\mt\in\Real^{d\times d}_{\mathrm{sym}}$, 
  $D(\mt)=1-f(\mt:\tenf C \mt)$, the damage model constitutive relation satisfies Assumption~\ref{ass:hypo}.
\end{example}

Before presenting the variational formulation of problem~\eqref{eq:ne.strong}, some useful notations are introduced. 
For $X\subset\overline{\Omega}$, we respectively denote by ${(\cdot,\cdot)}_X$ and $\nor{\cdot}_{X}$ the standard inner product and norm in $\Lt[X]$, with the convention that the subscript is omitted whenever $X=\Omega$. The same notation is used in the vector- and tensor-valued cases. $\Hod$ and $\Hdiv$ stand for the Sobolev spaces composed of vector-valued $\Ltd$ functions with weak gradient in $\Ltdd$, and tensor-valued $\Ltdd$ functions with weak divergence in $\Ltd$, respectively. 
Multiplying equation~\eqref{eq:ne.mech} by a test function $\vv\in\Hozd$ and integrating by 
parts one has
\begin{equation}
  \label{eq:ne.weak}
  (\msu,\GRADs{\vv})= (\vf,\vv). 
\end{equation}
Owing to the growth assumption~\eqref{eq:hypo.gr}, for all $\vv,\vw\in \Hod$, the form 
\begin{subequations}\label{eq:def_form_a}
  \begin{alignat}{1}
    & a(\vv,\vw)\eqbydef (\msu[\vv], \GRADs{\vw})
  \end{alignat}
\end{subequations}
is well defined and, from equation~\eqref{eq:ne.weak}, we can derive the following weak formulation of~\eqref{eq:ne.strong}:
\begin{equation}
  \label{eq:weak}
  \begin{aligned}
    \text{Given }\vf\in \Ltd,\text{ find }\vu\in \Hozd\text{  s.t., }  
    \forall\vv\in \Hozd,\;a(\vu,\vv) = (\vf,\vv).
  \end{aligned}
\end{equation}
 From \eqref{eq:weak} it is clear that the analytical stress tensor $\msu$ lies in the space $\Hdivs\eqbydef\{\mt\in\Ltdd ~|~ \DIV\mt\in\Ltd \text{ and $\mt$ is symmetric}\}$.  


\subsection{Discrete setting}\label{sec:discretization}

The discretization \eqref{eq:weak} is based on a conforming triangulation $\Th$ of $\Omega$, i.e.\ a set of closed triangles or tetrahedra with union $\overline{\Omega}$ and such that, for any distinct $T_1,T_2\in\Th$, the set $T_1\cap T_2$ is either a common edge, a vertex, the empty set or, if $d=3$, a common face. We assume that $\Th$ verifies the minimum angle condition, i.e., there exists $\alpha_{\mathrm{min}}>0$ uniform with respect to all considered meshes such that the minimum angle $\alpha_T$ of each $T\in\Th$ satisfies $\alpha_T\ge\alpha_{\mathrm{min}}$. 
The set of vertices of the mesh is denoted by $\Vh$; it is decomposed into interior vertices $\Vint$ and boundary vertices $\Vext$. For all $a\in\Vh$, $\Ta$ is the patch of elements sharing the vertex $a$, $\oma$ the corresponding open subdomain in $\Omega$ and $\mathcal V_a$ the set of vertices in $\oma$. For  all $T\in\Th$, $\VT$ denotes the set of vertices of $T$, $h_T$ its diameter and $\vec{n}_T$ its unit outward normal vector. 

For all $p\in\IN$ and all $T\in\Th$, we denote by $\IP[p](T)$ the space of $d$-variate polynomials in $T$ of total degree at most $p$ and by $\IP[p](\Th)=\{\varphi\in\Lt ~|~ \restrto{\varphi}{T}\in\IP[p](T) ~ \forall T\in\Th\}$ the corresponding broken space over $\Th$. In the same way we denote by $[\IP[p](T)]^d$ and $[\IP[p](T)]^{d\times d}$, respectively, the space of vector-valued and tensor-valued polynomials of total degree $p$ over $T$, and by $[\IP[p](\Th)]^d$ and $[\IP[p](\Th)]^{d\times d}$ the corresponding broken spaces over $\Th$.

In this work we will focus on conforming discretizations of problem \eqref{eq:ne.weak} of polynomial degree $p\ge2$ to avoid numerical locking, cf \cite{Vog83}. The discrete formulation reads: find $\uh\in \Hozd\cap[\IP[p](\Th)]^d$ such that 
\begin{equation}
  \label{eq:disc}
  \begin{aligned} 
    \forall\vh\in \Hozd\cap[\IP[p](\Th)]^d,\quad a(\uh,\vh) = (\vf,\vh).
  \end{aligned}
\end{equation}
This problem is usually solved using some iterative linearization algorithm defining at each iteration $k\ge1$ a linear approximation $\ms^{k-1}$ of $\ms$. Then the linearized formulation reads: find $\uhk\in \Hozd\cap[\IP[p](\Th)]^d$ such that 
\begin{equation}
  \label{eq:Newton} 
    \forall\vh\in \Hozd\cap[\IP[p](\Th)]^d,\quad (\mslinea,\GRADs\vh) = (\vf,\vh).
\end{equation}
For the Newton algorithm the linearized stress tensor is defined as
\begin{equation}
  \label{eq:sig_k-1}
  \mslinea \eqbydef \frac{\partial\ms(\mt)}{\partial\mt}\vert_{\mt=\GRADs\uhk[k-1]}\GRADs(\uhk-\uhk[k-1])+\msu[{\uhk[k-1]}].
\end{equation}


\section{Equilibrated stress reconstruction} \label{sec:reconst}

In general, the discrete stress tensor $\msuh$ resulting from \eqref{eq:disc} does not lie in $\Hdiv$ and thus cannot verify the equilibrium equation \eqref{eq:ne.mech}. In this section we will reconstruct from $\msuh$ a discrete stress tensor $\msh$ satisfying these properties. Based on this reconstruction, we then devise two equilibrated stress tensors representing the discrete stress and the linearization error respectively, which will be useful for the distinction of error components in the a posteriori error estimate of Section~\ref{sec:a_post_distinguishing}.


\subsection{Patchwise construction in the Arnold--Falk--Winther mixed finite element spaces} \label{sec:reconst_basic}

Let us for now suppose that $\uh$ solves~\eqref{eq:disc} exactly, before considering iterative linearization methods such as~\eqref{eq:Newton} in Section~\ref{sec:reconst_distinguishing}.
For the stress reconstruction we will use mixed finite element formulations on patches around mesh vertices in the spirit of \cite{fvca8, RieDPiErn17}. The mixed finite elements based on the dual formulation of \eqref{eq:ne.mech} will provide a stress tensor lying in $\Hdiv$. A global computation is too expensive for this post-processing reconstruction, so we solve local problems on patches of elements around mesh vertices. The goal is to obtain a stress tensor $\msh$ in a suitable (i.e. $H(\rm{div})$-conforming) finite element space by summing up these local solutions. The local problems are posed such that this global stress tensor is close to the discrete stress tensor $\msuh$ obtained from \eqref{eq:disc}, and that it satisfies the mechanical equilibrium on each element. 

In \cite{RieDPiErn17} the stress tensor is reconstructed in the Arnold--Winther finite element space \cite{ArnWin02}, directly providing symmetric tensors, but requiring high computational effort. In this work, as in \cite{fvca8}, we weaken the symmetry constraint and impose it weakly, as proposed in \cite{ArnFalWin07}:
for each element $T\in\Th$, the local Arnold--Falk--Winther mixed finite element spaces of degree $q\ge1$ hinge on the Brezzi--Douglas--Marini mixed finite element spaces \cite{BreDouMar86} for each line of the stress tensor and are defined by
\begin{subequations} 
 \label{eq:BDM.local}
\begin{align*}
\mS[T] &\eqbydef  [\IP[q](T)]^{d\times d}, \\
\vV[T]  &\eqbydef  [\IP[q-1](T)]^d,\\
\matr\Lambda_T &\eqbydef  \{ \matr{\mu} \in [\IP[q-1](T)]^{d\times d} ~\vert ~ \matr{\mu}=-\matr{\mu}^T \}.
\end{align*}
\end{subequations}
For $q=2$, the degrees of freedom are displayed in Figure~\ref{fig:AFW_T}.
On a patch $\oma$ the global space $\mS[h](\oma)$ is the subspace of $\Hdiv[\oma]$ composed of functions belonging piecewise to $\mS[T]$. The spaces $\vV[h](\oma)$ and $\matr\Lambda_h(\oma)$ consist of functions lying piecewise in $\vV[T]$ and $\matr\Lambda_T$ respectively, with no continuity conditions between two elements.

\begin{figure}[t]
  \centering
  \def\scale{0.7}
  \begin{tikzpicture}[scale=\scale]
    \def\cy{0.3}
    \def\cx{0.52}
    \def\ch{3.464}
    \def\cl{2}
    \def\cs{5}
    \draw (-\cl,0) -- (\cl,0) -- (0,\ch) -- (-\cl,0);

    \fill (-0.25*\cl -0.09,0.25*\ch) circle (2.2pt)
         (-0.25*\cl+0.09,0.25*\ch) circle (2.2pt);
    \fill (0.25*\cl-0.09,0.25*\ch) circle (2.2pt)
          (0.25*\cl+0.09,0.25*\ch) circle (2.2pt);
    \fill (-0.09,0.5*\ch) circle (2.2pt)
          (0.09,0.5*\ch) circle (2.2pt);

    \draw[-implies,double equal sign distance,thick] (0,0) -- (0,-2*\cy);
    \draw[-implies,double equal sign distance,thick] (-0.67*\cl,0) -- (-0.67*\cl,-2*\cy);
    \draw[-implies,double equal sign distance,thick] (0.67*\cl,0) -- (0.67*\cl,-2*\cy);

    \draw[-implies,double equal sign distance,thick] (-0.83*\cl,0.17*\ch) -- (-0.83*\cl - \cx,0.17*\ch + \cy);
    \draw[-implies,double equal sign distance,thick] (-0.5*\cl,0.5*\ch) -- (-0.5*\cl - \cx,0.5*\ch + \cy);
    \draw[-implies,double equal sign distance,thick] (-0.17*\cl,0.83*\ch) -- (-0.17*\cl - \cx,0.83*\ch + \cy);

    \draw[-implies,double equal sign distance,thick] (0.83*\cl,0.17*\ch) -- (0.83*\cl + \cx,0.17*\ch + \cy);
    \draw[-implies,double equal sign distance,thick] (0.5*\cl,0.5*\ch) -- (0.5*\cl + \cx,0.5*\ch + \cy);
    \draw[-implies,double equal sign distance,thick] (0.17*\cl,0.83*\ch) -- (0.17*\cl + \cx,0.83*\ch + \cy);

    \draw (-\cl+\cs,0) -- (\cl+\cs,0) -- (0+\cs,\ch) -- (-\cl+\cs,0);

    \fill (-0.25*\cl -0.09+\cs,0.25*\ch) circle (2.2pt)
         (-0.25*\cl+0.09+\cs,0.25*\ch) circle (2.2pt);
    \fill (0.25*\cl-0.09+\cs,0.25*\ch) circle (2.2pt)
          (0.25*\cl+0.09+\cs,0.25*\ch) circle (2.2pt);
    \fill (-0.09+\cs,0.5*\ch) circle (2.2pt)
          (0.09+\cs,0.5*\ch) circle (2.2pt);

    \draw (-\cl+2*\cs,0) -- (\cl+2*\cs,0) -- (0+2*\cs,\ch) -- (-\cl+2*\cs,0);

    \fill (-0.25*\cl+2*\cs,0.25*\ch) circle (2.2pt);
    \fill (0.25*\cl+2*\cs,0.25*\ch) circle (2.2pt);
    \fill (2*\cs,0.5*\ch) circle (2.2pt);
  \end{tikzpicture}
  \caption{Element diagrams for $(\matr{\Sigma}_T,\vec{V}_T,\matr\Lambda_T)$ in the case $d=q=2$}
  \label{fig:AFW_T}
\end{figure}
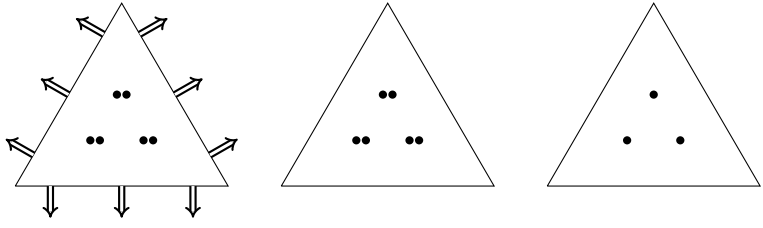
Let now $q\eqbydef p$. On each patch we need to consider subspaces where a zero normal component is enforced on the stress tensor on the boundary of the patch, so that the sum of the local solutions will have continous normal component across any mesh face inside $\Omega$. Since the boundary condition in the exact problem prescribes the displacement and not the normal stress, we distinguish the case whether $a$ is an interior vertex or a boundary vertex. If $a\in\Vint$ we set
\begin{subequations}
  \label{eq:space.BDM}
  \begin{eqnarray}
    \label{eq:space.BDM.S.int}
    \mS[h]^{a} &\eqbydef& \{\mth\in\mS[h](\oma)~|~
                      \mth\vn[\oma]=\vec 0 \text{~on~} \del\oma\},\\
    \label{eq:space.BDM.V.int}
    \vV[h]^{a} &\eqbydef& \{\vh\in\vV[h](\oma) ~|~ (\vh,\vz)_{\oma}=0 ~ \forall\vz\in\vec{RM}^d\},\\
    \label{eq:space.BDM.L.int}
    \matr\Lambda_h^{a} &\eqbydef& \matr\Lambda_h(\oma),
  \end{eqnarray}
where $\vec{RM}^2\eqbydef\{\vec b+c (x_2,-x_1)^T~\vert~\vec b\in\IR^2, c\in\IR\}$ and $\vec{RM}^3\eqbydef\{\vec b+\vec a \times\vec x~\vert~\vec b\in\IR^3, \vec{a}\in\IR^3\}$ are the spaces of rigid-body motions respectively for $d=2$ and $d=3$. If $a\in\Vext$ we set
  \begin{eqnarray}
    \label{eq:space.BDM.S.ext}
    \mS[h]^{a} &\eqbydef& \{\mth\in\mS[h](\oma)~|~
                      \mth\vn[\oma]=\vec 0 \text{~on~} \del\oma\backslash \del\Omega \},\\
    \label{eq:space.BDM.V.ext}
    \vV[h]^{a} &\eqbydef& \vV[h](\oma),\\
    \label{eq:space.BDM.L.ext}
    \matr\Lambda_h^{a} &\eqbydef& \matr\Lambda_h(\oma).
  \end{eqnarray}
\end{subequations}
For each vertex $a\in\Vh$ we define its hat function $\psi_a\in\IP[1](\Th)$ as the piecewise linear function taking value one at the vertex $a$ and zero on all other mesh vertices. 

\begin{const}[Stress tensor reconstruction]\label{const:BDM}
    Let $\uh$ solve \eqref{eq:disc}. For each $a\in\Vh$ find $(\sha,\vec r_h^a,\matr\lambda_h^a)\in \mS[h]^{a}\times\vV[h]^{a}\times\matr\Lambda_h^a$ such that for all $( \mth,\vh,\matr\mu_h)\in \mS[h]^{a}\times\vV[h]^{a}\times\matr\Lambda_h^a$,
    \begin{subequations}  \label{eq:stress_BDM_eq}
     \begin{eqnarray}
         (\sha,\mth)_{\oma} + (\vec{r}^a_h,\DIV\mth)_{\oma} + (\matr{\lambda}_h^a,\mth)_{\oma} &=& 
         (\psi_a\msuh,\mth)_{\oma} ,\label{eq:stress_BDM_eq1}
         \\
        (\DIV \sha,\vh)_{\oma} &=& (-\psi_a\vf + \msuh\vGRAD\psi_a,\vh)_{\oma},\label{eq:stress_BDM_eq2}
        \\
        (\sha,\matr{\mu}_h)_{\oma} & =& 0.\label{eq:stress_BDM_eq3}
       \end{eqnarray}
     \end{subequations}
   Then, extending $\sha$ by zero outside $\oma$, set $\msh\eqbydef\sum_{a\in\Vh}\sha$.
\end{const}
For interior vertices, the source term in \eqref{eq:stress_BDM_eq2} has to verify the Neumann compatibility condition 
\begin{equation}
    (-\psi_a\vf + \msuh\vGRAD\psi_a,\vz)_{\oma}=0 \quad \forall\vz\in\vec{RM}^d.
    \label{eq:hatfctorth}
\end{equation}
Taking $\psi_a\vz$ as a test function in \eqref{eq:disc}, we see that \eqref{eq:hatfctorth} holds and we obtain the following result.
\begin{lemma}[Properties of $\msh$] \label{lem:general_stress_properties}
Let $\msh$ be prescribed by Construction~\ref{const:BDM}. 
Then $\msh\in\Hdiv$, and for all $T\in\Th$, the following holds: 
\begin{equation}
     \label{eq:data_osc}
     (\vf+\DIV\msh,{\vv})_T =0 \quad\forall {\vv}\in\vV[T] ~~\forall T\in\Th.
\end{equation}

\end{lemma}
\begin{proof}
All the fields $\sha$ are in $\Hdiv[\oma]$ and satisfy appropriate zero normal conditions so that their zero-extension to $\Omega$ is in $\Hdiv$. Hence, $\msh\in\Hdiv$. Let us prove~\eqref{eq:data_osc}. 
Since \eqref{eq:hatfctorth} holds for all $a\in\Vint$, we infer that~\eqref{eq:stress_BDM_eq2} is actually true for all $\vh \in\vV[h](\oma)$. 
The same holds if $a\in\Vext$ by definition of $\vV[h]^a$. 
Since $\vV[h](\oma)$ is composed of piecewise polynomials that can be chosen independently in each cell $T\in\Ta$, and using $\msh\vert_T=\sum_{a\in\VT}\sha\vert_T$ and the partition of unity $\sum_{a\in\VT}\psi_a = 1$, we infer that $(\vf+\DIV\msh,{\vv})_T=0$ for all ${\vv}\in\vV[T]$ and all $T\in\Th$.
\end{proof}


\subsection{Discretization and linearization error stress reconstructions} \label{sec:reconst_distinguishing}
Let now, for $k\ge1$, $\uhk$ solve~\eqref{eq:Newton}. We will construct two different equilibrated $H(\rm{div})$-conforming stress tensors. The first one, $\mshdisc$, represents as above the discrete stress tensor $\msu[\uhk]$, for which we will have to modify Construction~\ref{const:BDM}, because the Neumann compatibility condition \eqref{eq:hatfctorth} is not satisfied anymore. The second stress tensor $\mshlin$ will be a measure for the linearization error and approximate $\mslinea-\msu[\uhk]$. The matrix resulting from the left side of \eqref{eq:stress_BDM_eq} will stay unchanged and we will only modify the source terms.

We denote by $\omsu$ the $L^2$-orthogonal projection of $\msu[\uhk]$ onto $[\IP[p-1](\Th)]^{d\times d}$ such that $(\msu[\uhk]-\omsu,\mth)=0$ for any $\mth\in[\IP[p-1](\Th)]^{d\times d}$.

\begin{const}[Discrete stress reconstruction]
   \label{const:disc}
   For each $a\in\Vh$ solve \eqref{eq:stress_BDM_eq} with $\uhk$ instead of $\uh$, $\omsu$ instead of $\msu[\uhk]$ and the source term in \eqref{eq:stress_BDM_eq2} replaced by
       $$-\psi_a\vf+\omsu\vGRAD\psi_a-\vy[{\rm disc}]^k,$$
   where $\vy[{\rm disc}]^k\in\vec{RM}^d$ is the unique solution of
     \begin{equation}
         \label{eq:ydisc}
         (\vy[{\rm disc}]^k,\vz)_{\oma} = -(\vf,\psi_a\vz)_{\oma} +( \omsu,\GRADs(\psi_a\vz))_{\oma}\quad\forall\vz\in\vec{RM}^d.
     \end{equation}
     The so obtained problem reads: find $(\sha,\vec r_h^a,\matr\lambda_h^a)\in \mS[h]^{a}\times\vV[h]^{a}\times\matr\Lambda_h^a$ such that for all $( \mth,\vh,\matr\mu_h)\in \mS[h]^{a}\times\vV[h]^{a}\times\matr\Lambda_h^a$,
    \begin{subequations}  
     \begin{eqnarray*} 
       (\sha,\mth)_{\oma} + (\vec{r}^a_h,\DIV\mth)_{\oma} + (\matr{\lambda}_h^a,\mth)_{\oma} &=& 
       (\psi_a\omsu,\mth)_{\oma} ,
       \\
       (\DIV \sha,\vh)_{\oma} &=& (-\psi_a\vf+\omsu\vGRAD\psi_a-\vy[{\rm disc}]^k,\vh)_{\oma},
       \\
       (\sha,\matr{\mu}_h)_{\oma} & =& 0.
      \end{eqnarray*}
     \end{subequations}
   Then set $\mshdisc\eqbydef\sum_{a\in\Vh}\sha$.
\end{const}

\begin{const}[Linearization error stress reconstruction]
   \label{const:lin}
   For each $a\in\Vh$ solve \eqref{eq:stress_BDM_eq} with $\uhk$ instead of $\uh$, the source term in \eqref{eq:stress_BDM_eq1} replaced by 
      $$\psi_a(\mslinea-\omsu),$$ 
   and the source term in \eqref{eq:stress_BDM_eq2} replaced by
      $$(\mslinea-\omsu)\vGRAD\psi_a+\vy[{\rm disc}]^k,$$
   where $\vy[{\rm disc}]^k\in\vec{RM}^d$ is defined by \eqref{eq:ydisc}. The corresponing local problem is to  find $(\sha,\vec r_h^a,\matr\lambda_h^a)\in \mS[h]^{a}\times\vV[h]^{a}\times\matr\Lambda_h^a$ such that for all $( \mth,\vh,\matr\mu_h)\in \mS[h]^{a}\times\vV[h]^{a}\times\matr\Lambda_h^a$,
    \begin{subequations} 
       \begin{eqnarray*}
         (\sha,\mth)_{\oma} + (\vec{r}^a_h,\DIV\mth)_{\oma} + (\matr{\lambda}_h^a,\mth)_{\oma} &=& 
         (\psi_a(\mslinea-\omsu),\mth)_{\oma} ,
         \\
        (\DIV \sha,\vh)_{\oma} &=& ((\mslinea-\omsu)\vGRAD\psi_a+\vy[{\rm disc}]^k,\vh)_{\oma},
        \\
        (\sha,\matr{\mu}_h)_{\oma} &=& 0.
       \end{eqnarray*}
     \end{subequations}
   Then set $\mshlin\eqbydef\sum_{a\in\Vh}\sha$.\\
\end{const}
Notice that the role of $\vy[{\rm disc}]^k$ is to guarantee that, for interior vertices, the source terms in Constructions~\ref{const:disc} and~\ref{const:lin} satisfy the Neumann compatibility conditions 
$$
  \begin{aligned}
  (-\psi_a\vf+\omsu\vGRAD\psi_a-\vy[{\rm disc}]^k,\vz)_{\oma} &=& 0 \quad \forall\vz\in\vec{RM}^d,
  \\
  ((\mslinea-\omsu)\vGRAD\psi_a+\vy[{\rm disc}]^k,\vz)_{\oma} &=& 0 \quad \forall\vz\in\vec{RM}^d.
  \end{aligned}
$$

\begin{lemma}[Properties of the discretization and linearization error stress reconstructions] 
 \label{lem:distinguishing_properties}
 Let $\mshdisc$ and $\mshlin$ be prescribed by Constructions~\ref{const:disc} and~\ref{const:lin}. Then it holds
 \begin{enumerate}
  \item $\mshdisc,\mshlin\in\Hdiv$,
  \item $(\vf + \DIV(\mshdisc+\mshlin),{\vv})_T = 0 \quad \forall {\vv}\in\vV[T] ~~\forall T\in\Th$,
  \item As the Newton solver converges, $\mshlin\to\matr0$.
 \end{enumerate}
\end{lemma}
\begin{proof}
  The proof is similar to the proof of Lemma~\ref{lem:general_stress_properties}. The first property is again satisfied due to the definition of $\mS[h]^a$.
  In order to show that the second property holds, we add the two equations \eqref{eq:stress_BDM_eq2} obtained for each of the constructions. The right hand side of this sum then reads $(-\psi_a\vf+\mslinea\vGRAD\psi_a,\vh)_{\oma}$. Once again we can, for any $\vz\in\vec{RM}^d$, take $\psi_a\vz$ as a test function in \eqref{eq:Newton} to show that this term is zero if $\vh\in\vec{RM}^d$, and so the equation holds for all $\vh\in\vV[h](\oma)$. Then we proceed as in the proof of Lemma~\ref{lem:general_stress_properties}.
\end{proof}


\section{A posteriori error estimate and adaptive algorithm}\label{sec:a_posteriori}

In this section we first derive an upper bound on the error between the analytical solution of~\eqref{eq:weak} and the solution $\uh$ of \eqref{eq:disc}, in which we then identify and distinguish the discretization and linearization error components at each Newton iteration for the solution $\uhk$ of \eqref{eq:Newton}. Based on this distinction, we present an adaptive algorithm stopping the Newton iterations once the linearization error estimate is dominated by the estimate of the discretization error. Finally, in a more theoretical part, we show the effectivity of the error estimate.


\subsection{Guaranteed upper bound}\label{sec:a_post_basic}
We measure the error in the energy norm
\begin{equation} 
   \label{eq:nor_en_def}
  \nor{\vv}_{\rm en}^2\eqbydef a(\vv,\vv) = (\msu[\vv],\GRADs\vv),
\end{equation}
for which we obtain the properties
\begin{equation}
   \label{eq:nor_en_prop}
   C_{\rm mon}^2 C_K^{-2}\nor{\GRAD\vv}^2 \le \nor{\vv}_{\rm en}^2 \le C_{\rm gro}\nor{\GRADs\vv}^2, 
\end{equation}
by applying~\eqref{eq:hypo.smon} and the Korn inequality for the left inequality, and the Cauchy--Schwarz inequality and~\eqref{eq:hypo.gr} for the right one. In our case it holds $C_K=\sqrt2$, owing to~\eqref{eq:ne.dirichlet}.

\begin{theorem}[Basic a posteriori error estimate]\label{thm:estimate_basic}
 Let $\vu$ be the analytical solution of \eqref{eq:weak} and $\uh$ the discrete solution of \eqref{eq:disc}. Let $\msh$ be the stress tensor defined in Construction~\ref{const:BDM}.
Then, 
\begin{equation}\label{eq:estimate_basic}
 \nor{\vu-\uh}_\mathrm{en} \le \sqrt2C_{\rm gro}C_{\rm mon}^{-3} \left(\sum_{T\in\Th} \big(\frac{h_T}{\pi}\nor{\vf+\DIV\msh}_T +  \nor{\msh-\msuh}_T  \big)^2\right)^{\nicefrac12}.
\end{equation}
\end{theorem}

\begin{remark}[Constants $C_{\rm gro}$ and $C_{\rm mon}$]
  For the estimate to be computable, the constants $C_{\rm gro}$ and $C_{\rm mon}$ have to be specified. 
  For the linear elasticity model \eqref{eq:LinCauchy} we set $C_{\rm gro} \eqbydef2\mu+d\lambda$ and $C_{\rm mon}\eqbydef\sqrt{2\mu}$, whereas for the Hencky--Mises model \eqref{eq:HenckyMises} we set $C_{\rm gro}\eqbydef2\tilde\mu(0)+d\tilde\lambda(0)$ and $C_{\rm mon}\eqbydef\sqrt{2\tilde\mu(0)}$. For the damage model \eqref{eq:Damage} we take $C_{\rm gro}\eqbydef C^*$ and $C_{\rm mon}\eqbydef\sqrt{C_*}$, where $C_*$ and $C^*$ are the constants appearing in \eqref{eq:unif.elliptic}. 
  Following \cite{fvca8}, we obtain a sharper bound in the case of linear elasticity, with $\mu^{-\nicefrac12}$ instead of $\sqrt2C_{\rm gro}C_{\rm mon}^{-3}$ in \eqref{eq:estimate_basic}.
\end{remark}

\begin{proof}[Proof of Theorem~\ref{thm:estimate_basic}]
We start by bounding the energy norm of the error by the dual norm of the residual of the weak formulation~\eqref{eq:weak}. 
Using \eqref{eq:nor_en_prop}, \eqref{eq:hypo.smon}, the linearity of $a$ in its second argument, and \eqref{eq:weak} we obtain
\begin{align*}
  \nor{\vu-\uh}_{\rm en}^2 
  &\le C_{\rm gro} \nor{\GRADs{(\vu-\uh)}}^2
  \le C_{\rm gro}C_{\rm mon}^{-2} \lvert a(\vu,\vu-\uh) -a(\uh,\vu-\uh)\rvert\\
  & = C_{\rm gro}C_{\rm mon}^{-2} \nor{\GRAD{(\vu-\uh)}}\left\vert a\left(\vu,\frac{\vu-\uh}{\nor{\GRAD{(\vu-\uh)}}}\right) -  a\left(\uh,\frac{\vu-\uh}{\nor{\GRAD{(\vu-\uh)}}}\right) \right\vert\\
  &\le C_{\rm gro} C_{\rm mon}^{-3}C_K\nor{\vu-\uh}_{\rm en} \sup_{\vv\in\Hozd, \nor{\GRAD{\vv}}=1} \{ a(\vu,\vv) - a(\uh,\vv) \}\\
  & = C_{\rm gro}C_{\rm mon}^{-3}C_K \nor{\vu-\uh}_{\rm en} \sup_{\vv\in\Hozd, \nor{\GRAD{\vv}}=1} \{ (\vf,\vv) - (\msuh,\GRADs{\vv}) \}.
\end{align*}
and thus
\begin{equation}
  \label{eq:nor_en_sup}
  \nor{\vu-\uh}_{\rm en}\le C_{\rm gro} C_{\rm mon}^{-3}C_K \sup_{\vv\in\Hozd, \nor{\GRAD{\vv}}=1} \{ (\vf,\vv) - (\msuh,\GRADs{\vv}) \}.
\end{equation}
Note that, due to the symmetry of $\ms$ we can replace $\GRADs\vv$ by $\GRAD\vv$ in the second term inside the supremum.
Now fix $\vv\in\Hozd$, such that $\nor{\GRAD{\vv}} = 1$. 
Since $\msh\in\Hdiv$, we can insert $(\DIV\msh,\vv)+(\msh,\GRAD{\vv})=0$ into the term inside the supremum and obtain
\begin{equation}
 \label{eq:split_res}
 (\vf,\vv)-(\msuh,\GRAD{\vv}) = (\vf+\DIV\msh,\vv)+(\msh-\msuh,\GRAD{\vv}).
\end{equation}
For the first term of the right hand side of \eqref{eq:split_res} we obtain, using \eqref{eq:data_osc} on each $T\in\Th$ to insert $\vec\Pi^{0}_T\vv$, which denotes the $\vec L^2$-projection of $\vv$ onto $[\IP[0](T)]^d$, the Cauchy--Schwarz inequality and the Poincar\'e inequality on simplices,
\begin{equation}
   \label{eq:proof_res}
    \bigl| (\vf+\DIV\msh,\vv) \bigr| 
\le  \Bigl| \sum_{T\in\Th} (\vf+\DIV\msh,\vv-\vec\Pi^{0}_T\vv)_T  \Bigr|
\le \sum_{T\in\Th}\frac{h_T}{\pi} \nor{\vf +\DIV\msh}_T\nor{\GRAD{\vv}}_T,
\end{equation}
whereas the Cauchy--Schwarz inequality applied to the second term directly yields
\begin{equation*}
     \bigl| (\msh-\msuh,\GRAD{\vv})  \bigr|
\le  \sum_{T\in\Th} \nor{\msh-\msuh}_T\nor{\GRAD{\vv}}_T.
\end{equation*}
Inserting these results in \eqref{eq:nor_en_sup} and again applying the Cauchy--Schwarz inequality yields the result.
\end{proof}


\subsection{Distinguishing the different error components}\label{sec:a_post_distinguishing}

The goal of this section is to elaborate the error estimate~\eqref{eq:estimate_basic} so as to distinguish different error components using the equilibrated stress tensors of Constructions~\ref{const:disc} and~\ref{const:lin}. 
This distinction is essential for the development of Algorithm \ref{algo:algo}, where the mesh and the stopping criteria for the iterative solver are chosen adaptively.

Notice that in Theorem~\ref{thm:estimate_basic} we don't necessarily need $\msh$ to be the stress tensor obtained in Construction~\ref{const:BDM}. We only need it to satisfy two properties: First, equation \eqref{eq:split_res} requires  $\msh$ to lie in $\Hdiv$. Second, in order to be able to apply the Poincar\'e inequality in \eqref{eq:proof_res}, $\msh$ has to satisfy the local equilibrium relation
\begin{equation}
  \label{eq:equilibrium}
 (\vf-\DIV\msh,{\vv})_T=0 \quad \forall {\vv}\in[\IP[0](T)]^d ~ \forall T\in\Th.
\end{equation}
Thus, the theorem also holds for $\msh\eqbydef\mshdisc+\mshlin$, where $\mshdisc$ and $\mshlin$ are defined in Constructions~\ref{const:disc} and~\ref{const:lin} and we obtain the following result.

\begin{theorem}[A posteriori error estimate distinguishing different error sources]\label{thm:estimate_distinguishing}
 Let $\vu$ be the analytical solution of \eqref{eq:weak}, $\uhk$ the discrete solution of \eqref{eq:Newton}, and $\msh\eqbydef\mshdisc+\mshlin$.
Then, 
\begin{equation}\label{eq:estimate_distinguishing}
 \nor{\vu-\uhk}_\mathrm{en} \le\sqrt2C_{\rm gro}C_{\rm mon}^{-3}
 \big( \eta_{{\rm disc}}^{k}+\eta_{{\rm lin}}^{k}+\eta_{{\rm quad}}^{k}+\eta_{{\rm osc}}^{k} \big),
\end{equation}
where the local discretization, linearization, quadrature and oscillation error estimators on each $T\in\Th$ are defined as
\begin{subequations}\label{eq:etas}
\begin{align}
 \eta_{{\rm disc},T}^{k} &\eqbydef \nor{\mshdisc-\omsu}_T,  \label{eq:eta_disc} \\
 \eta_{{\rm lin},T}^{k}    &\eqbydef \nor{\mshlin}_T,  \label{eq:eta_lin} \\
 \eta_{{\rm quad},T}^{k}&\eqbydef \nor{\omsu-\msu[\uhk]}_T,  \label{eq:eta_quad} \\
 \eta_{{\rm osc},T}^{k}  &\eqbydef \frac{h_T}{\pi}\nor{\vf-\vec\Pi^{p-1}_T\vf}_T, \label{eq:eta_osc}
\end{align}
\end{subequations}
with $\vec\Pi^{p-1}_T$ denoting the $\vec L^2$-projection onto $[\IP[p-1](T)]^d$, and for each error source the global estimator is given by
\begin{equation}
 \label{eq:eta_glob}
 \eta_{\cdot}^k \eqbydef \Big( 4\sum_{T\in\Th} (\eta_{\cdot,T}^k)^2\Big)^{\nicefrac12}.
\end{equation}
\end{theorem}
\begin{proof}
  Using $\msh\eqbydef\mshdisc+\mshlin$ in Theorem~\ref{thm:estimate_basic}, we obtain 
\begin{equation*}
 \nor{\vu-\uhk}_\mathrm{en} \le \sqrt2C_{\rm gro}C_{\rm mon}^{-3} \left(\sum_{T\in\Th} \big(\frac{h_T}{\pi}\nor{\vf+\DIV(\mshdisc+\mshlin)}_T +  \nor{\mshlin+\mshdisc-\msu[\uhk]}_T  \big)^2\right)^{\nicefrac12}.
\end{equation*}
  Applying the second property of Lemma~\ref{lem:distinguishing_properties} in the first term yields the oscillation error estimator. In the second term we add and substract $\omsu$ and apply the triangle inequality to obtain 
\begin{equation*}
    \nor{\vu-\uhk}_\mathrm{en} \le\sqrt2C_{\rm gro}C_{\rm mon}^{-3} \left(\sum_{T\in\Th} \big( 
 \eta_{{\rm disc},T}^{k}+\eta_{{\rm lin},T}^{k}+\eta_{{\rm quad},T}^{k}+\eta_{{\rm osc},T}^{k} \big)^2\right)^{\nicefrac12}.
\end{equation*}
Owing to \eqref{eq:eta_glob}, the previous bound yields the conclusion.
\end{proof}

\begin{remark}[Quadrature error]
In practice, the projection $\omsu$ of $\msu[\uhk]$ onto $[\IP[p-1](\Th)]^{d\times d}$ for a general nonlinear stress-strain relation cannot be computed exactly. 
The quadrature error estimator $\eta_{{\rm quad},T}^{k}$ measures the quality of this projection.
\end{remark}


\subsection{Adaptive algorithm}
\label{sec:algorithm}

Based on the error estimate of Theorem~\ref{thm:estimate_distinguishing}, we propose an adaptive algorithm where the mesh size is locally adapted, and a dynamic stopping criterion is used for the linearization iterations.
The idea is to compare the estimators for the different error sources with each other in order to concentrate the computational effort on  reducing the dominant one.
For this purpose, let $\gamma_{\mathrm{lin}},\gamma_{\mathrm{quad}}\in(0,1)$, be user-given weights and $\Gamma>0$ a chosen threshold that the error should not exceed. 

\begin{algo}[Adaptive algorithm] \label{algo:algo} 
\end{algo}
    \textbf{Mesh adaptation loop}
    \begin{enumerate}
        \item Choose an initial function $\uh^0\in\Hozd\cap[\IP[p](\Th)]^d$ and set $k\eqbydef1$
        \item Set the initial quadrature precision $\nu\eqbydef2p$ (exactness for polynomials up to degree $\nu$)
        \item \textbf{Linearization iterations}
        \begin{enumerate}
           \item Calculate $\mslinea$, $\uh^{k}$, $\msu[\uhk]$ and $\omsu$
           \item Calculate the stress reconstructions $\mshdisc$ and $\mshlin$ and the error estimators 
                 $\eta_{\rm{disc}}^k$, $\eta_{\rm{lin}}^k$, $\eta_{\rm{osc}}^k$ and $\eta_{\rm{quad}}^k$
           \item Improve the quadrature rule (setting $\nu\eqbydef\nu+1$) and go back to step $3$(a) until
             \begin{alignat}{1} 
                \label{eq:crit_quad}
               \eta_{\rm{quad}}^k \le \gamma_{\rm{quad}}(\eta_{\mathrm{disc}}^k+\eta_{\mathrm{lin}}^{k}+\eta_{\mathrm{osc}}^k)
             \end{alignat}
           \item \textbf{End of the linearization loop} if  
            \begin{alignat}{1}
               \label{eq:crit_lin}
              \eta_{\mathrm{lin}}^{k} \leq \gamma_{\mathrm{lin}}(\eta_{\mathrm{disc}}^{k}+\eta_{\mathrm{osc}}^k)            
            \end{alignat}
        \end{enumerate}
        \item Refine or coarsen the mesh $\Th$ such that the local discretization error estimators 
              $\eta_{{\mathrm{disc}},T}^k$ are distributed evenly
    \end{enumerate}
    \textbf{End of the mesh adaptation loop} if $\eta_{\mathrm{disc}}^k + \eta_{\mathrm{osc}}^k \leq  \Gamma$\\
   
Instead of using the global stopping criteria \eqref{eq:crit_quad} and \eqref{eq:crit_lin}, which are evaluated over all mesh elements, we can also define the local criteria
\begin{subequations}\label{eq:crit_local}
 \begin{align}
  \eta_{\mathrm{quad},T}^k &\le \gamma_{\mathrm{quad}}(\eta_{\mathrm{disc},T}^k+\eta_{\mathrm{lin},T}^{k}+\eta_{\mathrm{osc},T}^k) 
  &\forall T\in\Th,
  \\
  \eta_{\mathrm{lin},T}^{k}&\le\gamma_{\mathrm{lin}}(\eta_{\mathrm{disc},T}^{k}+\eta_{\mathrm{osc},T}^k) &\forall T\in\Th,       
 \end{align}
\end{subequations}
where it is also possible to define local weights $\gamma_{\mathrm{lin},T}$ and $\gamma_{\mathrm{quad},T}$ for each element. These local stopping criteria are necessary to establish the local efficiency of the error estimators in the following section, whereas the global criteria are only sufficient to prove global efficiency.


\subsection{Local and global efficiency}\label{sec:efficiency}

To prove efficiency, we will use a posteriori error estimators of residual type. Following \cite{Ver96,Ver99} we define for $X\subseteq\Omega$ the functional $\mathcal R_X:\Hod[X]\to \vec H^{-1}(X)$ such that, for all $\vv\in\Hod[X],\vw\in\Hozd[X]$,
$$
   \langle \mathcal R_X(\vv),\vw \rangle_X \eqbydef (\msu[\vv],\GRADs{\vw})_X - (\vf,\vw)_X.
$$
In what follows we let $a\lesssim b$ stand for $a\le C b$ with a generic constant C, which is independent of the mesh size, the domain $\Omega$ and the stress-strain relation, but that can depend on the shape regularity of the mesh family $\{\Th\}_h$ and on the polynomial degree $p$.

Define, for each $T\in\Th$, 
\begin{equation}\label{eq:residual_estimators}
  \begin{aligned}
  (\eta_{\sharp,T}^{k})^2 &\eqbydef h_T^2 \nor{\DIV\omsu[\uhk]+\vec\Pi^{p}_T\vf}_{T}^2
  +\sum_{F\in\Fh[T]^{\rm{i}}}h_F \nor{\jump{ \omsu[\uhk]\vec n_F} }_F^2,
  \\
  (\eta_{\flat,T}^{k})^2 &\eqbydef h_T^2 \nor{\DIV(\msu[\uhk]-\omsu[\uhk])}_{T}^2 
  +\sum_{F\in\Fh[T]^{\rm{i}}}h_F \nor{\jump{(\msu[\uhk]-\omsu[\uhk])\vec n_F} }_F^2.
  \end{aligned}
\end{equation}
The quantity $\eta_{\flat,T}^{k}$ obviously measures the quality of the approximation of $\msu[\uhk]$ by $\omsu[\uhk]$ and can be estimated explicitly.
The following result is shown in \cite[Section 3.3]{Ver96}.
Denoting for any $T\in\Th$ by $\Th[T]$ the set of elements sharing an edge (if $d=2$) or a face ($d=3$) with $T$, it holds
\begin{equation}\label{eq:Verfuerth_pre}
 \eta_{\sharp,T}^{k}\lesssim \nor{\mathcal R_{\Th[T]}(\uhk)}_{\vec H^{-1}(\Th[T])} 
 +\Bigl( \sum_{T'\in\Th[T]} (\eta_{\flat,T'}^{k}+\eta_{{\rm osc},T'}^k)^2\Bigr)^{\nicefrac12}.
\end{equation}
In order to bound the dual norm of the residual, we need an additional assumption on the stress-strain relation which, in particular, implies the growth assumption \eqref{eq:hypo.gr}.
\begin{assumption}[Stress-strain relation II]
  \label{ass:hypo_add}
There exists a real number $C_{\rm Lip}\in(0,+\infty)$ such that, for all $\matr{\tau},\matr{\eta}\in\Real^{d\times d}_{\mathrm{sym}}$,
  \label{eq:hypo_add}
  \begin{equation} 
    \label{eq:hypo.lip}
    \norm[d\times d]{\ms(\matr{\tau})-\ms(\matr{\eta})}\le C_{\rm Lip}\norm[d\times d]{\matr{\tau}-\matr{\eta}}.
    \quad \text{(Lipschitz continuity)}
  \end{equation}
\end{assumption}
Notice that the three stress-strain relations presented in Examples \ref{ex:lin.cauchy}, \ref{ex:hencky}, and \ref{ex:damage} satisfy the previous Lipschitz continuity assumptions.
Owing to the definition of the functional $\mathcal R_{\Th[T]}$ and to the fact that $-\DIV\msu=\vf\in \vec L^2(\Th[T])$, using the Cauchy--Schwarz inequality and the Lipschitz continuity~\eqref{eq:hypo.lip} of $\ms$, it is inferred that
$$
\begin{aligned}
\nor{\mathcal R_{\Th[T]}(\uhk)}_{\vec H^{-1}(\Th[T])} &\eqbydef 
\sup_{\vw\in\vec H^{-1}(\Th[T]),\;\nor{\vw}_{\vec H_0^1(\Th[T])}\le 1} 
(\msu[\uhk],\GRADs\vw)_{\Th[T]} -(\vf,\vw)_{\Th[T]} 
\\
&=\sup_{\vw\in\vec H^{-1}(\Th[T]),\;\nor{\vw}_{\vec H_0^1(\Th[T])}\le 1} 
(\msu[\uhk]-\msu,\GRADs\vw)_{\Th[T]}
\\
&\le\sup_{\vw\in\vec H^{-1}(\Th[T]),\;\nor{\vw}_{\vec H_0^1(\Th[T])}\le 1}
\nor{\msu[\uhk]-\msu}_{\Th[T]} \nor{\GRADs\vw}_{\Th[T]}
\\
&\le C_{\rm Lip}\nor{\GRADs(\vu-\uhk)}_{\Th[T]}.
\end{aligned}
$$
Thus, by \eqref{eq:Verfuerth_pre}, the previous bound, and the strong monotonicity \eqref{eq:hypo.smon} it holds
\begin{equation}\label{eq:Verfuerth}
 \eta_{\sharp,T}^{k}\lesssim C_{\rm Lip}C_{\rm mon}^{-1}\nor{\vu-\uhk}_{\rm{en},\Th[T]}+\eta_{\flat,\Th[T]}^{k}+\eta_{\rm{osc},\Th[T]}^k,
\end{equation}
where $\eta_{\cdot,\Th[T]}^{k}\eqbydef \bigl\{2\sum_{T'\in\Th[T]}(\eta_{\cdot,T'}^{k})^2\bigr\}^{\nicefrac12}$.

\begin{theorem}[Local efficiency]\label{thm:local_efficiency}
 Let $\vu\in\Hozd$ be the solution of \eqref{eq:weak}, $\uhk\in\Hozd\cap[\IP[p](\Th)]^d$ be arbitrary and $\mshdisc$ and $\mshlin$ defined by Constructions~\ref{const:disc} and \ref{const:lin}. Let the local stopping criteria \eqref{eq:crit_local} be verified. Then it holds for all $T\in\Th$,
\begin{equation}\label{eq:local_efficiency}
 \eta_{{\rm disc},T}^{k}+\eta_{{\rm lin},T}^{k} +\eta_{{\rm quad},T}^{k} +\eta_{{\rm osc},T}^{k}  \lesssim C_{\rm Lip}C_{\rm mon}^{-1}\nor{\vu-\uhk}_{\rm{en},\Th[T]}
+\eta_{\flat,\Th[T]}^{k} +\eta_{\mathrm{osc},\Th[T]}^{k}.
\end{equation}
\end{theorem}
 It is well known that there exist nonconforming finite element methods which are equivalent to mixed finite element methods using the Brezzi--Douglas--Marini spaces (see e.g. \cite{ArbChe95}). 
 Following the ideas of~\cite{ErnVoh13,HanSteVoh12,ElAErnVoh11} and references therein, we use these spaces to prove Theorem~\ref{thm:local_efficiency}.
We will denote by $\vM[h](\oma)$ the extension to vector valued functions of the nonconforming space introduced in \cite{ArbChe95} on a patch $\oma$. Recall that $\mS[h](\oma)$ is the subspace of $\Hdiv$ containing tensor-valued piecewise polynomials of degree at most $p$. 

For $d=2$, the space $\vM[T]$ on a triangle $T\in\Th$ is given by
 \begin{equation}\label{eq:MT}
 \vM[T]\eqbydef \left\{ \begin{array}{rll} 
    \{ \vv\in[\IP[p+2](T)]^d &\vert~ \restrto{\vv}{F}\in[[\IP[p+1](F)]^d ~~\forall F\in\Fh[T]\} & \text{if $p$ is even,} \\
    \{ \vv\in[\IP[p+2](T)]^d &\vert~ \restrto{\vv}{F}\in[\IP[p](F)]^d\oplus \tvIP[p+2](F)~ ~\forall F\in\Fh[T]\} & \text{if $p$ is odd,}
  \end{array} \right.
 \end{equation} 
 where $\tvIP[p+2](F)$ is the $\vec L^2(F)$--orthogonal complement of $[\IP[p+2](F)]^d$ in $[\IP[p+1](F)]^d$. The degrees of freedom are given by the moments up to degree $(p-1)$ inside each $T\in\Th$ and up to degree $p$ on each edge $F\in\Fh$. On a patch $\omega_a$ this means that $\vM[h](\oma)$ contains vector-valued functions lying piecewise in $\vM[T]$ such that
 \begin{equation}
   \label{eq:Mh_prop1}
   (\jump{\vmh},\vh)_F = 0 \quad \forall F\in\Fh[a]\setminus\Fh^{\rm{ext}} ~~ \forall \vh\in[\IP[p](F)]^d,
 \end{equation}
 where $\Fh[a]$ contains the faces in $\Fh$ to which $a$ belongs, and $\Fh^{\rm{ext}}$ the faces lying on $\partial\Omega$. We will denote by $\vM[h]^a$ the subspace of $\vM[h](\oma)$ with functions $\vmh$ verifying
 \begin{equation} \label{eq:Mh_prop2int}
   (\vmh,\vz)_{\oma} = 0 \quad \forall \vz\in\vec{RM}^d,
 \end{equation}
 if $a\in\Vint$, and
 \begin{equation} \label{eq:Mh_prop2ext}
   ( \vmh, \vh )_F = 0 \quad \forall F\in\Fh[a]\cap\Fh^{\rm{ext}} ~~\forall \vh\in[\IP[p](F)]^d,
 \end{equation}
 if $a\in\Vext$.

 We will use the space $\vM[h]^a$ together with Proposition~\ref{prop:rha} to prove Theorem~\ref{thm:local_efficiency}. 
 For Proposition~\ref{prop:rha} we introduce two equivalent formulations of Construction~\ref{const:disc} based on the following spaces
 \begin{align}
    \tmST&\eqbydef \{\mt\in\mS[T]~\vert~ (\mt,\matr\mu)_{T}=0 ~ \forall\matr\mu\in\matr\Lambda_T\},\\
    \tilde{\matr\Sigma}_{h}(\oma) &\eqbydef \{\mth\in\Ltdd~\vert~ \mth\in\tmST~\forall T\in\Ta\},\\
    \tmS &\eqbydef \mS[h]^a\cap\tilde{\matr\Sigma}_h(\oma) = \{\mth\in\mS[h]^a~\vert~ (\mth,\matr\mu_h)_{\oma}=0 ~ \forall\matr\mu_h\in\matr\Lambda_h^a\},\\
    \vec L_h^a &\eqbydef \{\vec l_h\in[\IP[p](\Fh[\oma])]^d ~\vert~\begin{array}[t]{l} \vec l_h = \vec 0 \text{ on }\partial\oma\text{ if } a \in\Vint, \\  \vec l_h = \vec 0 \text{ on }\partial\oma\setminus\partial\Omega\text{ if } a \in\Vext\},\end{array}
\end{align}
 where $\Fh[\oma]$ collects the faces of the patch.
 The first equivalent formulation of Construction~\ref{const:disc} consists in finding $\sha\in\tmS$ and $\vec r_h^a\in\vV[h]^a$ such that for all $(\mth,\vh)\in\tmS\times\vV[h]^a$
    \begin{subequations}  \label{eq:stress_equiv}
     \begin{eqnarray}
        (\sha,\mth)_{\oma} + (\vec{r}^a_h,\DIV\mth)_{\oma} &=& (\psi_a\omsu ,\mth)_{\oma} ,\label{eq:stress_equiv1}
        \\
        (\DIV \sha,\vh)_{\oma} &=& (-\psi_a\vf + \omsu\vGRAD\psi_a - \vy[{\rm disc}]^k,\vh)_{\oma}.\label{eq:stress_equiv2}
       \end{eqnarray}
     \end{subequations}
 The second formulation is the first step when hybridizing the mixed problem \eqref{eq:stress_equiv}. Following \cite{ArbChe95} it consists in using the broken space $ \tilde{\matr\Sigma}_{h}(\oma)$ instead of $\tmS$ and imposing the continuity of the normal stress components by Lagrange multipliers. Its solution is $(\sha,\vec r_h^a,\vec l_h^a )\in\tilde{\matr\Sigma}_{h}(\oma)\times\vV[h]^a\times\vec L_h^a$ such that for all $(\mth,\vh,\vec l_h)\in\tilde{\matr\Sigma}_{h}(\oma)\times\vV[h]^a\times\vec L_h^a$
   \begin{subequations}  \label{eq:stress_equiv_hyb}
     \begin{eqnarray}
       &&(\sha,\mth)_{\oma}+\sum_{T\in\Th[a]}(\vec{r}^a_h,\DIV\mth)_{T}
       -\sum_{F\in\Fh[\oma]}(\vec l_h^a,\jump{\mth\vn[F]}_F)_F = (\psi_a\omsu ,\mth)_{\oma} ,\label{eq:stress_equiv_hyb1}
       \\
       &&\sum_{T\in\Th[a]}(\DIV \sha,\vh)_{T}
       = (-\psi_a\vf+\omsu\vGRAD\psi_a-\vy[{\rm disc}]^k,\vh)_{\oma}, \label{eq:stress_equiv_hyb2}
       \\
       &&- \sum_{F\in\Fh[\oma]} (\jump{\sha\vn[F]}_F,\vec l_h)_F = 0,\label{eq:stress_equiv_hyb3}
     \end{eqnarray}
   \end{subequations}
  where we denote by $\vn[TF]$ the outward normal vector of $T$ on $F$ and by $\vn[F]$ the normal vector of $F$ with an arbitrary, but fixed direction.
  In particular, \eqref{eq:stress_equiv_hyb1} can be reformulated as 
  \begin{equation}\label{eq:stress_equiv_hyb_T}
    (\sha-\psi_a\omsu ,{\mt}_T)_T + (\vec{r}^a_h,\DIV{\mt}_T)_{T} = \sum_{F\in\Fh[T]} (\vec l_h^a,{\mt}_T\vn[TF])_F 
    \quad \forall {\mt}_T\in\tmST~\forall T\in\Ta.
  \end{equation}

\begin{proposition}\label{prop:rha}
  Let $a\in\Vh$ and let $(\sha,\vec r_h^a,\vec l_h^a)\in\tilde{\matr\Sigma}_{h}(\oma)\times\vV[h]^a\times\vec L_h^a$ be defined by \eqref{eq:stress_equiv_hyb}. 
  Let $\tilde{\vec r}_h^a$ be a vector-valued function verifying for all $T\in\Th[a]$ and for all $F\in\Fh[T]$,
  \begin{subequations}
  \label{eq:r_tilde}
   \begin{align}
     \restrto{\tilde{\vec r}_h^a}{T} & \in \vM[T],\\
     \vec\Pi_{\vec L_F}\restrto{\tilde{\vec r}_h^a}{F} &= \restrto{\vec l_h^a}{F}, \label{eq:r_tilde_1} \\
     \vec\Pi_{\vV[T]} \restrto{\tilde{\vec r}_h^a}{T} &=\restrto{\vec r_h^a}{T}, \label{eq:r_tilde_2}
   \end{align}
  \end{subequations}
  where $\vec\Pi_{\vec L_F}$  and $\vec\Pi_{\vV[T]}$ denote, respectively, the $L^2$-projections on $\vec L_F = [\IP[p](F)]^d$ and $\vV[T] = [\IP[p-1](T)]^d$. Then $\tilde{\vec r}_h^a\in\vM[h]^a$.
\end{proposition}

\begin{proof}
  From $\mathrm{dim}(\vV[T])+3\mathrm{dim}(\vec L_F) = p^2+p+3(2p+2)=p^2+7p+6=\mathrm{dim}(\vM[T])$ we infer that problem \eqref{eq:r_tilde} is well-posed.
  Plugging \eqref{eq:r_tilde_1} and \eqref{eq:r_tilde_2} into \eqref{eq:stress_equiv_hyb_T} yields 
  \begin{equation}
  \matr\Pi_{\tmST} \restrto{(\GRAD\tilde{\vec r}_h^a)}{T}= \restrto{(\sha-\psi_a\omsu)}{T}. \label{eq:r_tilde_3}
  \end{equation}
  Since the formulations \eqref{eq:stress_equiv} and \eqref{eq:stress_equiv_hyb} are equivalent, we can insert \eqref{eq:r_tilde_3} and \eqref{eq:r_tilde_2} into \eqref{eq:stress_equiv1} and obtain
  \begin{equation*}
    (\GRAD\tilde{\vec r}_h^a,\mth)_{\oma}+(\tilde{\vec r}_h^a,\DIV\mth)_{\oma}=0 \quad\forall\mth\in\tmS.
  \end{equation*}
  Choosing a basis function of $\tmS$ having zero normal trace across all edges except one edge $F$ and applying the Green theorem, we see that $\tilde{\vec r}_h^a$ satisfies \eqref{eq:Mh_prop1} for faces $F\in\Fh[a]\setminus\Fh^{\rm ext}$, since the normal components across $F$ of a basis of $\tmST$ span $[\IP[p](F)]^d$.
  If $a\in\Vext$ we can proceed in the same way for $F\in\Fh[a]\cap\Fh^{\rm{ext}}$ to obtain $\eqref{eq:Mh_prop2ext}$. Finally, for $a\in\Vint$ it holds $(\vec r_h^a,\vz)_{\oma}=0$ for any $\vz\in\vec{RM}^d$ by the definition \eqref{eq:space.BDM.V.int} of $\vV[h]^a$, and by \eqref{eq:r_tilde_2} it follows that $\tilde{\vec r}_h^a$ satisfies \eqref{eq:Mh_prop2int}.  We conclude that $\tilde{\vec r}_h^a$ lies in $\vM[h]^a$.
\end{proof}
\begin{proof}[Proof of Theorem~\ref{thm:local_efficiency}]
  We start by proving the local approximation property of the discrete stress reconstruction for any $T\in\Th$
  \begin{equation}
    \label{eq:local_approximation}
    \eta_{{\rm disc},T}^k = \nor{\mshdisc-\omsu}_T\lesssim\eta_{\sharp,\Th[T]}^{k} +\eta_{\mathrm{osc},\Th[T]}^{k}.
  \end{equation}
  We define $\tilde{\vec r}_h^a$ by \eqref{eq:r_tilde}. Then using the fact that $\tilde{\vec r}_h^a\in\vM[h]^a$ by Proposition~\ref{prop:rha} and \cite[Lemma 5.4]{Voh10}, stating that the dual norm on $\vM[h]$ is an upper bound for the $\vec H^1$-seminorm, we obtain
  \begin{equation} \label{eq:etadisc_vs_sup}
   \begin{aligned}
    \nor{\sha-\psi_a\omsu}_{\oma}
    \le \nor{\GRAD\tilde{\vec r}_h^a}_{\oma} 
    \lesssim \sup_{\vec m_h\in\vec M_h^a, \nor{\GRAD\vec m_h} = 1}  (\sha - \psi_a \omsu,\GRAD\vec m_h)_{\oma}.
   \end{aligned}
  \end{equation}
  Now fix $\vec m_h\in\vec M_h^a$ such that $\nor{\GRAD\vec m_h}_{\oma} = 1$. Then, by \eqref{eq:Mh_prop1}, it follows
  \begin{align*}
    (\sha - \psi_a&\omsu,\GRAD\vec m_h)_{\oma} \\
    & = \sum_{T\in\Th[a]} (\sha -\psi_a\omsu,\GRAD\vec m_h)_T \\
    & = \underbrace{-\sum_{T\in\Th[a]} (\DIV\sha-\DIV(\psi_a\omsu,\vec m_h)_T}_{\fedybqe\term_1}
      + \underbrace{\sum_{F\in\Fh[a]} (\jump{\psi_a\omsu\vn[F]},\vec m_h)_F}_{\fedybqe\term_2}.
  \end{align*}
  Using \eqref{eq:stress_equiv2} (which, as in the proof of Lemma~\ref{lem:distinguishing_properties}, is valid for all $\vh\in\vV[h](\oma)$) and the fact that for all $T\in\Th[a]$ and $\mt\in\mS[T]$ it holds $(\DIV\mt,\vec m_h)_T=(\DIV\mt,\Pi_{\vV[T]}\vec m_h)_T$, due to the property $\DIV\mS[T]=\vV[T]$, we can write for the first term
  \begin{align*}
    \term_1 &= -\sum_{T\in\Th[a]}  (-\psi_a\vf + \omsu\vGRAD\psi_a -\DIV(\psi_a\omsu),\vec\Pi_{\vV[T]}\vec m_h)_T \\
   &= -\sum_{T\in\Th[a]} (\psi_a(\vf+\DIV\omsu),\vec\Pi_{\vV[T]}\vec m_h)_T\\
   &= -\sum_{T\in\Th[a]} (\vec\Pi^{p}_T\vf+\DIV\omsu,\psi_a\vec\Pi_{\vV[T]}\vec m_h)_T\\
   &\le \left( \sum_{T\in\Th[a]} h_T^2\nor{\psi_a(\vec\Pi^{p}_T\vf+\DIV(\omsu))}_T^2\right)^{\nicefrac12}  \left(\sum_{T\in\Th[a] }h_T^{-2}\nor{\vec m_h}_T^2\right)^{\nicefrac12} \\
   &\lesssim \left( \sum_{T\in\Th[a]}  h_T^2 \nor{\vec\Pi^{p}_T\vf + \DIV\omsu}_T^2 \nor{\psi_a}_{L^{\infty}(T)}^2\right)^{\nicefrac12}  \nor{\GRAD \vec m_h}_{\oma},
  \end{align*}
  where we used the Cauchy-Schwarz, the discrete Poincar\'e inequality of \cite[Theorem 8.1]{Voh05} together with \eqref{eq:Mh_prop2int} if $a\in\Vint$ and the discrete Friedrichs inequality of \cite[Theorem 5.4]{Voh05} together with \eqref{eq:Mh_prop2ext} if $a\in\Vext$, and $\nor{\psi_a}_{L^{\infty}(T)}=1$. For the second term we proceed in a similar way, using the discrete trace inequality $\nor{\vec m_h}_F\lesssim h_F^{-\nicefrac12}\nor{\vec m_h}_T$, and obtain
  \begin{align*}
   \term_2 &= \sum_{F\in\Fh[a]^{{\rm int}}}  (\psi_a\jump{\omsu\vn[F]},\vec m_h)_F 
   \\
   &\le \left( \sum_{F\in\Fh[a]^{{\rm int}}}  h_F \nor{\psi_a\jump{\omsu\vn[F]}}_F^2\right)^{\nicefrac12}\left( \sum_{F\in\Fh[a]^{{\rm int}}}  h_F^{-1}\nor{\vec m_h}_F^2\right)^{\nicefrac12} 
   \\
   &\lesssim \left(  \sum_{F\in\Fh[a]^{{\rm int}}}  h_F \nor{\jump{\omsu\vn[F]}}_F^2 \right)^{\nicefrac12} \nor{\GRAD\vec m_h}_{\oma}.
  \end{align*}
  Inserting these results into \eqref{eq:etadisc_vs_sup} yields \eqref{eq:local_approximation}.
  
  From the local stopping criteria \eqref{eq:crit_local}, the definition of the discretization error estimator \eqref{eq:eta_disc} and the local approximation property \eqref{eq:local_approximation} it follows that 
  $$\eta_{{\rm disc},T}^{k}+\eta_{{\rm lin},T}^{k}+\eta_{{\rm quad},T}^{k}  \lesssim \eta_{\mathrm{disc},T}^{k} = \nor{\mshdisc -\omsu}_T   \lesssim \eta_{\sharp,\Th[T]}^{k}+\eta_{\mathrm{osc},\Th[T]}^{k}.$$
  Then \eqref{eq:Verfuerth} yields the result. 
\end{proof}

\begin{theorem}[Global efficiency]
Let $\vu\in\Hozd$ be the solution of \eqref{eq:weak}, $\uhk\in\Hozd\cap[\IP[p](\Th)]^d$ be arbitrary and $\mshdisc$ and $\mshlin$ defined by Constructions~\ref{const:disc} and \ref{const:lin}. Let the stopping criteria \eqref{eq:crit_quad} and \eqref{eq:crit_lin} be verified. Then it holds
\begin{equation}\label{eq:global_efficiency}
 \eta_{{\rm disc}}^{k}+\eta_{{\rm lin}}^{k} +\eta_{{\rm quad}}^{k} +\eta_{{\rm osc}}^{k}  \lesssim C_{\rm Lip}C_{\rm mon}^{-1}\nor{\vu-\uhk}_{\rm{en}}
+\eta_{\flat}^{k} +\eta_{\mathrm{osc}}^{k}.
\end{equation}
\end{theorem}
\begin{proof}
Proceeding as above, using the global stopping criteria \eqref{eq:crit_quad} and \eqref{eq:crit_lin}, and owing to \eqref{eq:local_approximation} we obtain 
\begin{align}
  \eta_{{\rm disc}}^{k}+\eta_{{\rm lin}}^{k} +\eta_{{\rm quad}}^{k} +\eta_{{\rm osc}}^{k}  
  \lesssim  \eta_{{\rm disc}}^{k}+\eta_{{\rm osc}}^{k}  
  \lesssim\Big( 2\sum_{T\in\Th}(\eta_{\sharp,\Th[T]}^k+\eta_{{\rm osc},\Th[T]}^k)^2\Big)^{\nicefrac12} + \eta_{\rm osc}^k 
  \lesssim \eta_{\sharp}^k+\eta_{{\rm osc}}^k.
\end{align}
Then, using again \eqref{eq:Verfuerth} we obtain the result.
\end{proof}


\section{Numerical results} \label{sec:tests}
In this section we illustrate numerically our results on two test cases, both performed with the {\tt Code\_Aster}\footnote{\url{http://web-code-aster.org}} software, which uses conforming finite elements of degree $p=2$.
Our intention is, first, to show the relevance of the discretization error estimators used as mesh refinement indicators,  and second, to propose a stopping criterion for the Newton iterations based on the linearization error estimator.
All the triangulations are conforming, since in the remeshing progress hanging nodes are removed by bisecting the neighboring element.

\subsection{L-shaped domain}

\begin{figure}[t]
 \includegraphics[width=0.32\textwidth]{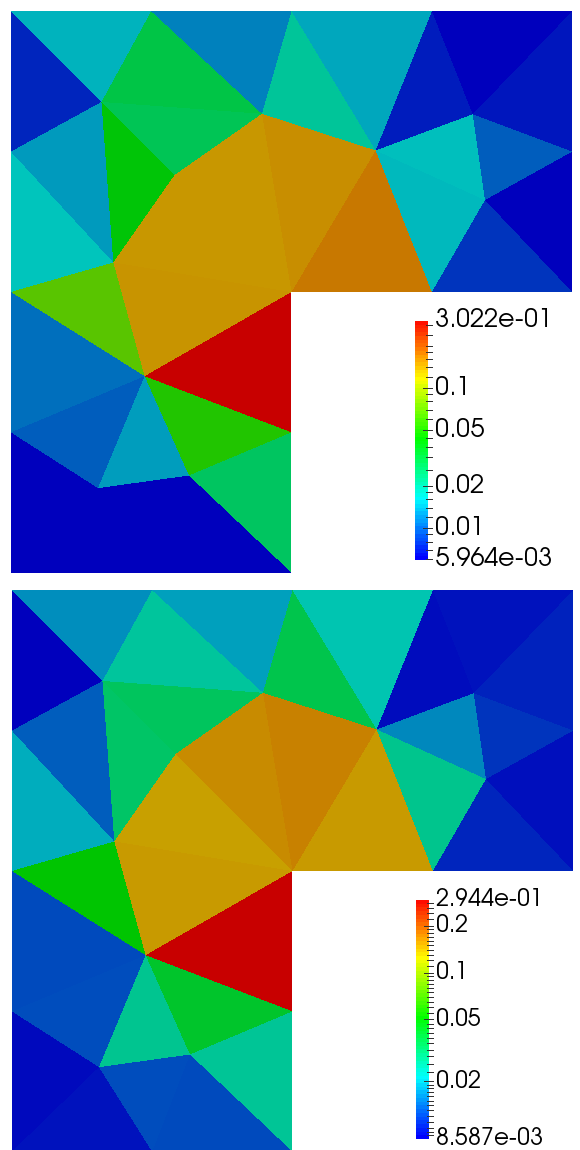}
 \includegraphics[width=0.32\textwidth]{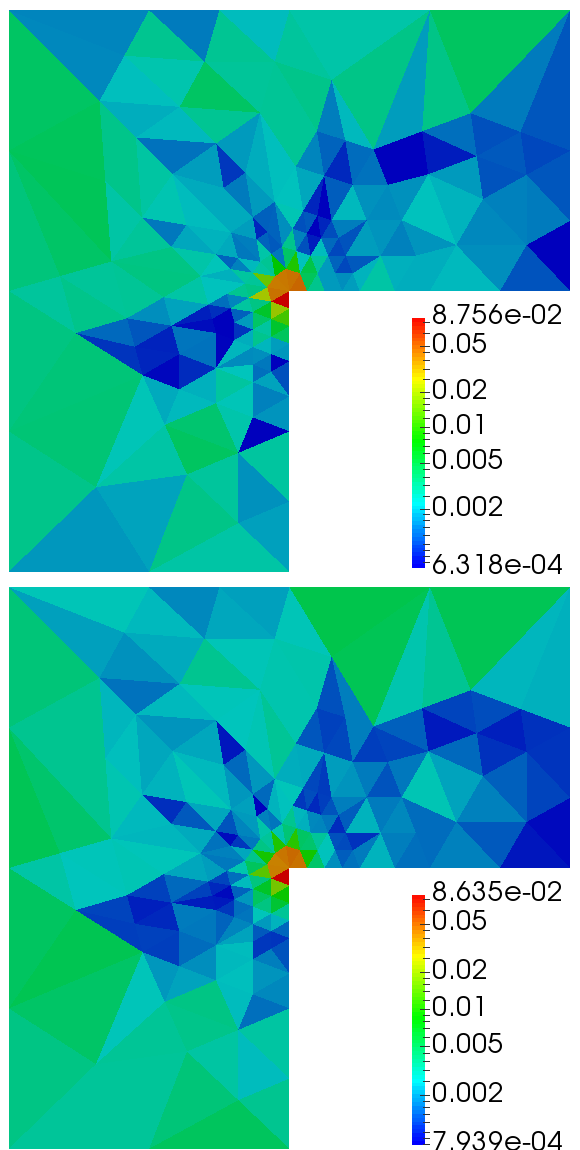}
 \includegraphics[width=0.32\textwidth]{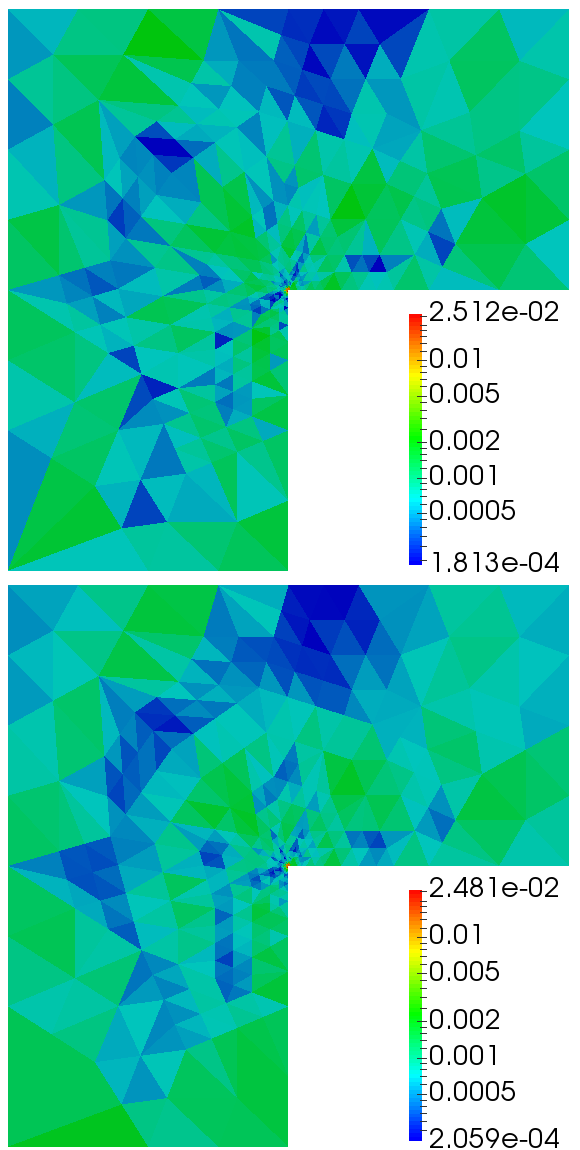}
 \caption{L-shaped domain with linear elasticity model. Distribution of the error estimators (top) and the analytical error (bottom) for the initial mesh (left) and after three (middle) and six (right) adaptive mesh refinements.}
 \label{fig:L_lin_distribution}
\end{figure}

Following \cite{Kim11,NicWitWoh08,AinRan10}, we consider the L-shaped domain $\Omega = (-1,1)^2\setminus([0,1]\times[-1,0])$
, where for the linear elasticity case an analytical solution is given by 
\begin{equation*}
  \vu(r,\theta) = \frac1{2\mu} r^{\alpha} \begin{pmatrix} \cos(\alpha\theta)-\cos((\alpha-2)\theta)\\ A\sin(\alpha\theta)+\sin((\alpha-2)\theta) \end{pmatrix},
\end{equation*}
with the parameters
$$ \mu = 1.0, \quad \lambda = 5.0, \quad \alpha = 0.6, \quad A = 31/9. $$
This solution is imposed as Dirichlet boundary condition on $\del\Omega$, together with $\vf = \vec0$ in $\Omega$. We perform this test for two different stress-strain relations. First on the linear elasticity model \eqref{eq:LinCauchy}, where we can compare the error estimate \eqref{eq:estimate_basic} to the analytical error $\nor{\vu-\uh}_{\rm{en}}$. The second relation is the nonlinear Hencky--Mises model \eqref{eq:HenckyMises}, for which we distinguish the discretization and linearization error components and use the adaptive algorithm from Section~\ref{sec:algorithm}.

\subsubsection{Linear elasticity model}

We compute the analytical error and its estimate on two series of unstructered meshes. Starting with the same initial mesh, we use uniform mesh refinement for the first one and adaptive refinement based on the error estimate for the second series. 

\begin{figure}[t]
 \centering
  \raisebox{-.55\height}{
 \begin{tikzpicture}
 \begin{axis}[
       xlabel = {$\vert\Th\vert = $ number of mesh elements }, 
       legend entries = {{error, unif.},{estimate, unif.},{error, adap.},{estimate, adap.}},
       legend pos = south west,
       xmode=log,
       ymode=log,
      ]
   \addplot [Sepia,mark = *] table [x = nbmai, y = err] {tests/L-shape/L-shape_ana_uni.dat};
   \addplot [RedOrange,mark = triangle*] table [x = nbmai, y = est] {tests/L-shape/L-shape_ana_uni.dat};
   \addplot [blue,mark = square*] table [x = nbmai, y = err] {tests/L-shape/L-shape_ana.dat};
   \addplot [LimeGreen,mark = diamond*] table [x = nbmai, y = est] {tests/L-shape/L-shape_ana.dat};
 \end{axis}
 \end{tikzpicture}}
 \quad
\begin{tabular}{|c|c|}
 \hline
 $\vert\Th\vert$  & $I_{\rm{eff}}$\\
 \hline
34&		1.00\\
84&		1.01\\
\rowcolor[gray]{0.9} 130&	1.02\\ 
137&	1.01\\
\rowcolor[gray]{0.9} 214&	1.05\\ 
239&	1.01\\
293&	1.00\\
429&	1.01\\
\rowcolor[gray]{0.9} 524&	1.01\\ 
601&	1.01\\
801&	1.01\\
1099&	1.02\\
\rowcolor[gray]{0.9} 1142&	1.02\\ 
 \hline
\end{tabular}
\caption{L-shaped domain with linear elasticity model. \emph{Left:} Comparison of the error estimate \eqref{eq:estimate_basic} and $\nor{\vu-\uh}_{\rm{en}}$ on two series of meshes, obtained by uniform and adaptive remeshing. \emph{Right:} Effectivity indices of the estimate for each mesh, with the meshes stemming from uniform refinement highlighted in gray.}
 \label{fig:L_lin_comp}
\end{figure}
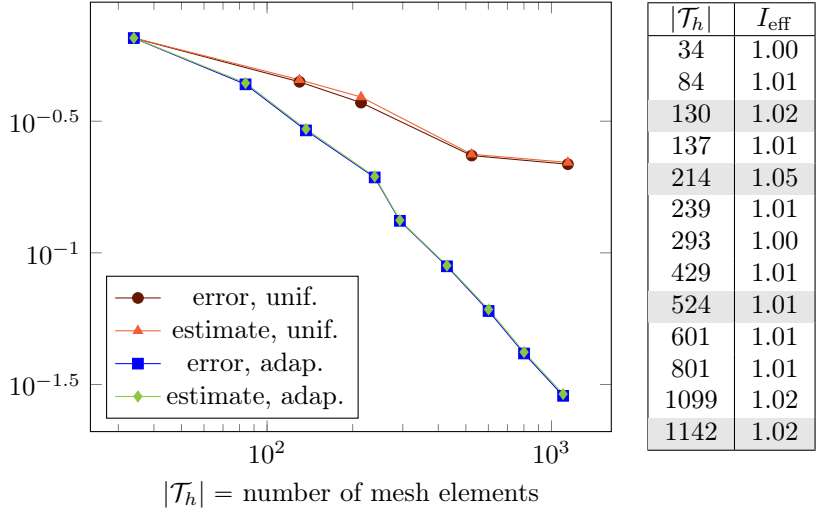

Figure~\ref{fig:L_lin_distribution} compares the distribution of the error and the estimators on the initial and two adaptively refined meshes. The error estimators reflect the distribution of the analytical error, which makes them a good indicator for adaptive remeshing.
Figure~\ref{fig:L_lin_comp} shows the global estimates and errors for each mesh, as well as their effectivity index corresponding to the ratio of the estimate to the error. We obtain effectivity indices close to one, showing that the estimated error value lies close to the actual one, what we can also observe in the graphics on the left.
As expected, the adaptively refined mesh series has a higher convergence rate, with corresponding error an order of magnitude lower for $10^3$ elements.

\subsubsection{Hencky--Mises model}\label{sec:L.HM} 

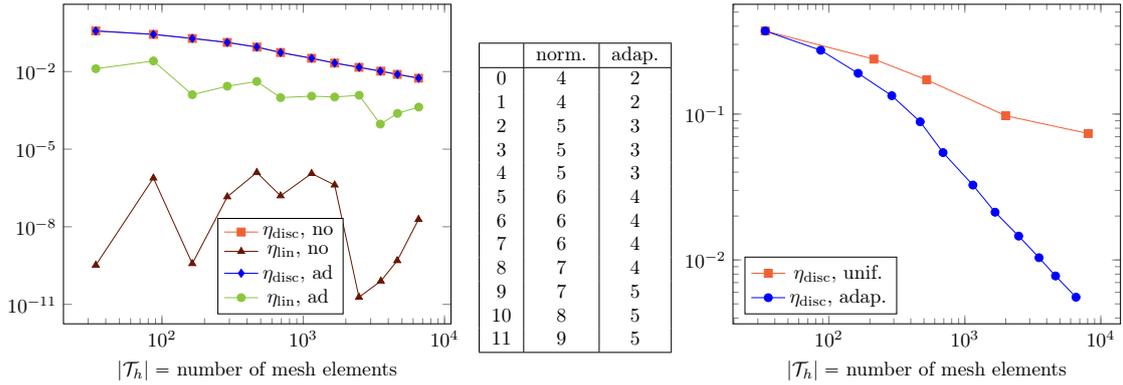
\begin{figure}[t]
 \centering
\resizebox {\columnwidth} {!} {
 \begin{tikzpicture}
 \begin{axis}[
       xlabel = {$\vert\Th\vert = $ number of mesh elements }, 
       legend entries = {{$\eta_{\rm{disc}}$, no},{$\eta_{\rm{lin}}$, no},{$\eta_{\rm{disc}}$, ad},{$\eta_{\rm{lin}}$, ad}},
       legend style={at={(0.56,0.03)},anchor=south},
       xmode=log,
       ymode=log,
      ]
   \addplot [RedOrange ,mark = square*] table [x = nbmare, y = eta_sp_resi] {tests/L-shape/L-shape_HM.dat};
   \addplot [Sepia,mark = triangle*] table [x = nbmare, y = eta_ln_resi] {tests/L-shape/L-shape_HM.dat};
   \addplot [blue,mark = diamond*] table [x = nbmari, y = eta_sp_rita] {tests/L-shape/L-shape_HM.dat};
   \addplot [LimeGreen,mark = *] table [x = nbmari, y = eta_ln_rita] {tests/L-shape/L-shape_HM.dat};
 \end{axis}
 \end{tikzpicture}
  \raisebox{1.15\height}{
\begin{tabular}{|c|c|c|}
 \hline
 & norm. & adap. \\
 \hline
0 & 4 & 2 \\
1 & 4 & 2 \\
2 & 5 & 3 \\
3 & 5 & 3 \\
4 & 5 & 3 \\
5 & 6 & 4 \\
6 & 6 & 4 \\
7 & 6 & 4 \\
8 & 7 & 4 \\
9 & 7 & 5 \\
10 & 8 & 5 \\
11 & 9 & 5 \\
 \hline
\end{tabular}}
 \begin{tikzpicture}
 \begin{axis}[
       xlabel = {$\vert\Th\vert = $ number of mesh elements }, 
       legend entries = {{$\eta_{\rm{disc}}$, unif.},{$\eta_{\rm{disc}}$, adap.}},
       legend pos = south west,
       xmode=log,
       ymode=log,
      ]
   \addplot [RedOrange ,mark = square*] table [x = nbmaun, y = eta_sp_unif] {tests/L-shape/L-shape_HM_remaillage.dat};
   \addplot [blue,mark = *] table [x = nbmaad, y = eta_sp_adap] {tests/L-shape/L-shape_HM_remaillage.dat};
 \end{axis}
 \end{tikzpicture}
}
\caption{L-shaped domain with Hencky--Mises model. \emph{Left:} Comparison of the global discretization and linearization error estimators on a series of meshes, without and with adaptive stopping criterion for the Newton algorithm. \emph{Middle:} Number of Newton iterations without and with adaptive stopping criterion for each mesh. \emph{Right:} Discretization error estimate for uniform and adaptive remeshing.}
 \label{fig:L_HM_comp}
\end{figure}

For the Hencky--Mises model we choose the Lam\'e functions
$$ \tilde{\mu}(\rho) \eqbydef a + b(1+\rho^2)^{\nicefrac{-1}2}, \quad \tilde{\lambda}(\rho) \eqbydef \kappa-\frac32 \tilde{\mu}(\rho),$$
corresponding to the Carreau law for elastoplastic materials (see, e.g.~\cite{GatMarRud13,LouGue90,San93}), and we set $a=1/20$, $b=1/2$, and $\kappa=17/3$ so that the shear modulus reduces progressively to approximately $10\%$ of its initial value. This model allows us to soften the singularity observed in the linear case and to validate our results on more homogeneous error distributions. We apply Algorithm~\ref{algo:algo} with $\gamma_{\rm{lin}} = 0.1$ and compare the obtained results to those without the adaptive stopping criterion for the Newton solver. In both cases, we use adaptive remeshing based on the spatial error estimators. 

The results are shown in Figure~\ref{fig:L_HM_comp}. In the left graphic we observe that the linearization error estimate in the adaptive case is much higher than in the one without adaptive stopping criterion. We see that this does not affect the discretization error estimator. The table shows the number of performed Newton iterations for both cases. The algorithm using the adaptive stopping criterion is more efficient. In the right graphic we compare the global distretization error estimate on two series of meshes, one refined uniformly and the other one adaptively, based on the local discretization error estimators. Again the convergence rate is higher for the adaptively refined mesh series.

\subsection{Notched specimen plate}
In our second test we use the two nonlinear models of Examples~\ref{ex:hencky} and~\ref{ex:damage} on a more application-oriented test. The idea is to set a special sample geometry yielding to a model discrimination test, namely different physical results for different models. 
We simulate the uniform traction of a notched specimen under plain strain assumption (cf. Figure~\ref{fig:plaque}). The notch is meant to favor strain localization phenomenon. We consider a domain $\Omega = (0,10{\rm m})\times(-10{\rm m},10{\rm m}) \setminus \{\vx\in\IR^2~\vert~ \nor{\vx{\rm m}- (0,11{\rm m})^T}\le2{\rm m}\}$, we take $\vf=\vec0$, and we prescribe a displacement on the boundary leading to the following Dirichlet conditions:
$$ 
u_x=0{\rm m} \text{~ if  } x = 0{\rm m}, \quad u_y=-1.1\cdot10^{-3}{\rm m}\text{~ if } y = -10{\rm m}, \quad u_y=1.1\cdot10^{-3}{\rm m} \text{~ if } y = 10{\rm m}.
$$ 
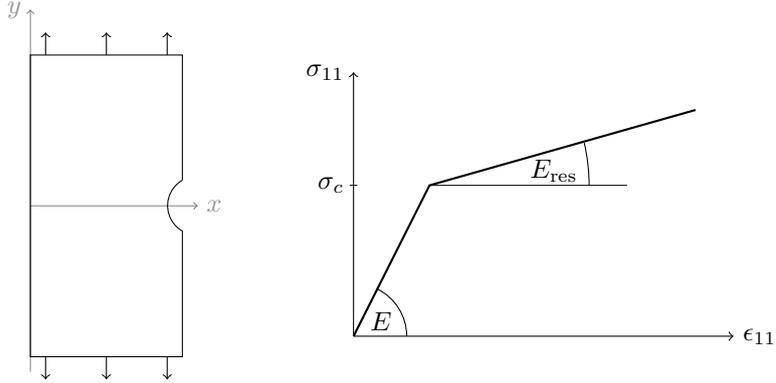
\begin{figure}[t]
\centering
  \raisebox{-.07\height}{
\begin{tikzpicture}[scale=0.2]
\draw[gray,->,very thin] (0,0) -- (11,0) node [right]{\textcolor{Gray}{$x$}};
\draw[gray,->,very thin] (0,-11) -- (0,13) node [left]{\textcolor{Gray}{$y$}};

\draw (10,1.7) -- (10,10) -- (0,10) -- (0,-10) -- (10,-10) -- (10,-1.7);
\draw (10,1.7) arc (119.97:240.93:1.945cm);

\draw[->] (1,-10) -- (1,-11.5);
\draw[->] (5,-10) -- (5,-11.5);
\draw[->] (9,-10) -- (9,-11.5);

\draw[->] (1,10) -- (1,11.5);
\draw[->] (5,10) -- (5,11.5);
\draw[->] (9,10) -- (9,11.5);
\end{tikzpicture}}
\hskip5ex
\begin{tikzpicture}[scale=1]
\draw[->] (0,0) -- (5,0) node [right]{$\epsilon_{11}$};
\draw[->] (0,0) -- (0,3.5) node [left]{$\sigma_{11}$};

\draw[thick] (0,0) -- (1,2) -- (4.5,3);
\draw[thin] (-0.05,2) -- (0.05,2);
\node[left] at (0,2){$\sigma_c$};
\node[right] at (0.1,0.2){$E$};
\draw[thin] (0.7,0) arc (0:64:0.7);
\draw[very thin] (1,2) -- (3.6,2);
\node[right] at (2.2,2.2){$E_{\rm{res}}$};
\draw[thin] (3.1,2) arc (0:13:2.5);
\end{tikzpicture}
\caption{\emph{Left:} Notched specimen plate. \emph{Right:} Uniaxial traction curve}
\label{fig:plaque}
\end{figure}%
In many applications, the information about the material properties are obtained in uniaxial experiments, yielding a relation between $\sigma_{ii}$ and $\epsilon_{ii}$ for a space direction $x_i$. Since we only consider isotropic materials, we can choose $i=1$. From this curve one can compute the nonlinear Lam\'e functions of \eqref{eq:HenckyMises} and the damage function in \eqref{eq:Damage}. Although the uniaxial relation is the same, the resulting stress-strain relations will be different.
In our test, we use the $\sigma_{11}$\,--\,$\epsilon_{11}$\,--\,relation indicated in the right of Figure~\ref{fig:plaque} with $$\sigma_c =3\cdot10^4\rm{Pa},  \quad E = \frac{\mu(3\lambda+2\mu)}{\lambda+\mu}=3\cdot10^8\rm{Pa}, \quad E_{\rm{res}} = 3\cdot10^7\rm{Pa},$$ corresponding to the Lam\'e parameters $\mu = \frac{3}{26}\cdot10^9\rm{Pa}$ and $\lambda =\frac9{52}\cdot10^9\rm{Pa}$.
\begin{figure}[t]
\resizebox {\columnwidth} {!} {
 \includegraphics[height=0.3\textheight]{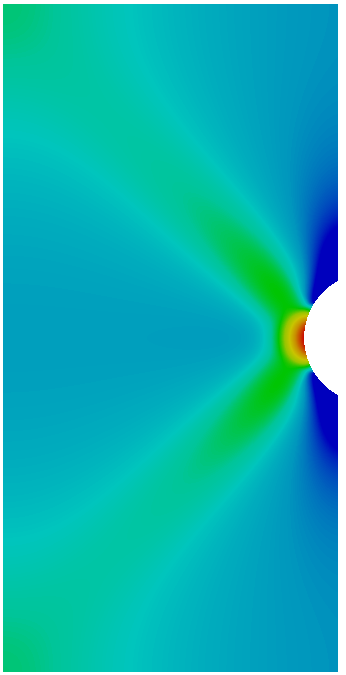}
 \includegraphics[height=0.3\textheight]{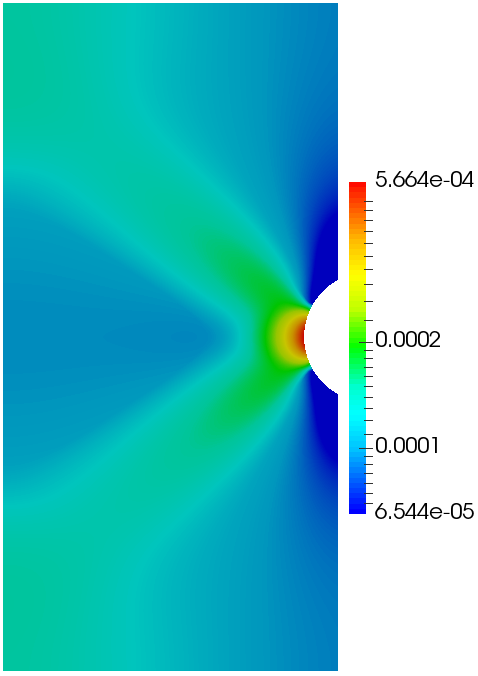}
\hspace{0.02\columnwidth}
 \includegraphics[height=0.3\textheight]{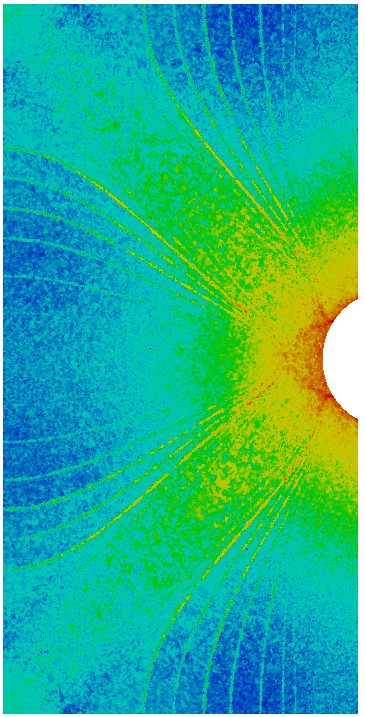}
 \includegraphics[height=0.3\textheight]{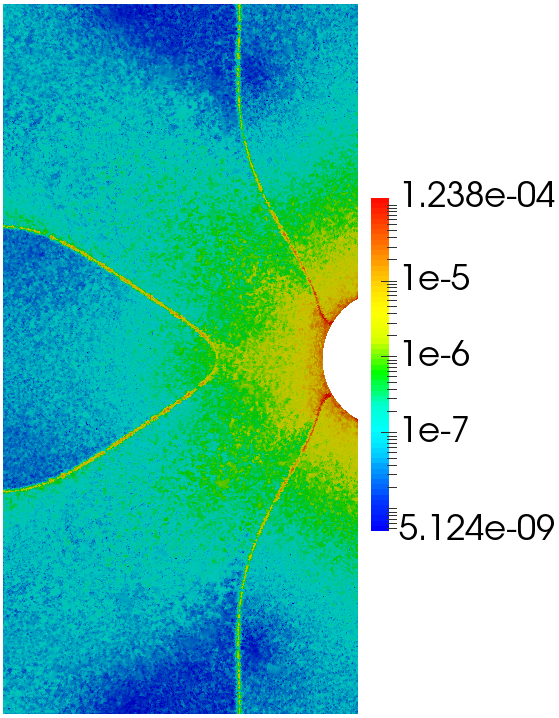}
}\\
\resizebox {\columnwidth} {!} {
 \includegraphics[height=0.3\textheight]{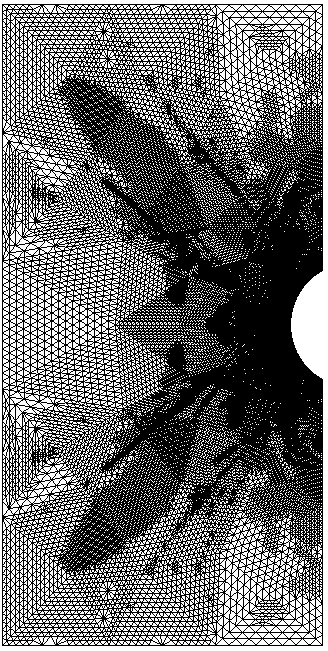}
 \includegraphics[height=0.3\textheight]{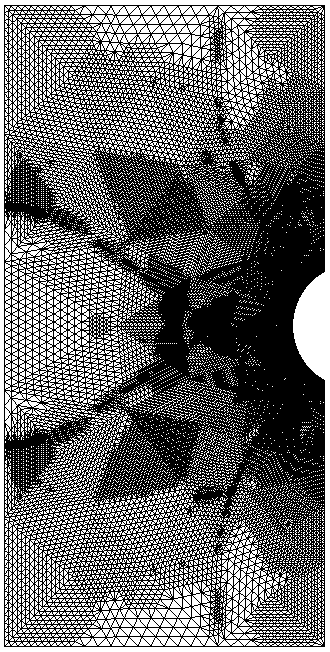}
 \includegraphics[height=0.13\textheight]{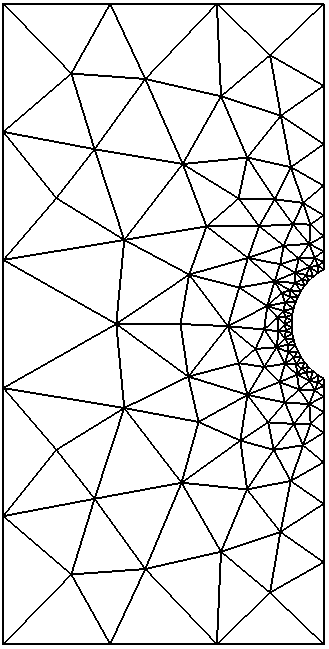}
\hspace{0.003\columnwidth}
 \includegraphics[height=0.3\textheight]{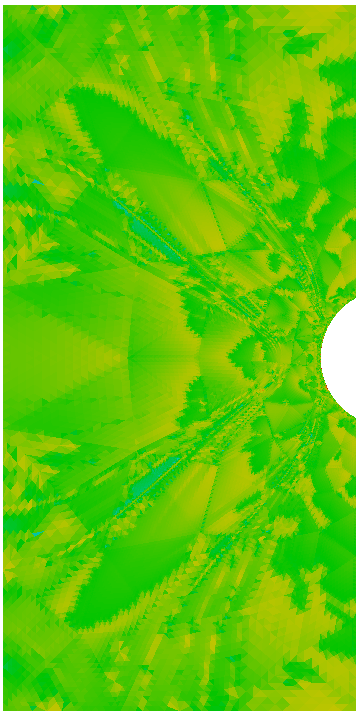}
 \includegraphics[height=0.3\textheight]{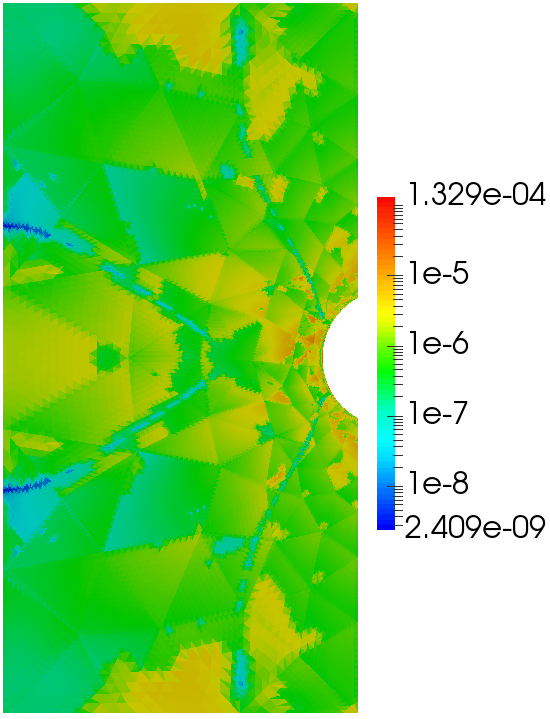}
}
 \caption{Notched specimen plate, comparison between Hencky--Mises (left in each picture) and damage model (right). \emph{Top left:} $\optr(\GRADs\uh)$. \emph{Top right:} $\eta_{\rm{disc}}$ on a fine mesh (no adaptive refinement). \emph{Bottom left:} meshes after six adaptive refinements. \emph{Bottom middle:} initial mesh. \emph{Bottom right:} $\eta_{\rm{disc}}$ on the adaptively refined meshes.}
 \label{fig:plate_img}
\end{figure}%
For both stress-strain relations we apply Algorithm~\ref{algo:algo} with $\gamma_{\rm{lin}}=0.1$. We first compare the results to a computation on a very fine mesh to evaluate the remeshing based on the discretization error estimators. Secondly, we perform adaptive remeshing based on these estimators but without applying the adaptive stopping of the Newton iterations and compare the two series of meshes. 
As in Section~\ref{sec:L.HM}, we verify if the reduced number of iterations impacts the discretization error.

Figure~\ref{fig:plate_img} shows the result of the first part of the test. In each of the four images the left specimen corresponds to the Hencky--Mises and the right to the isotropic damage model. 
To illustrate the difference of the two models, the top left picture shows the trace of the strain tensor. This scalar value is a good indicator for both models, representing locally the relative volume increase which could correspond to either a damage or shear band localization zone. 
In the top right picture we see the distribution of the discretization error estimators in the reference computation (209,375 elements), whereas the distribution of the estimators on the sixth adaptively refined mesh is shown in the bottom right picture (60,618 elements for Hencky--Mises, 55,718 elements for the damage model). 
The corresponding meshes and the initial mesh for the adaptive algorithm are displayed in the bottom left of the figure. 
To ensure a good discretization of the notch after repeated mesh refinement, the initial mesh cannot be too corse in this curved area. 
We observe that the adaptively refined meshes match the distribution of the discretization error estimators on the uniform mesh, and that the estimators are more evenly distributed on these meshes.

The results of the second part of the test are illustrated in Figure~\ref{fig:plaque_Newton}. As for the L-shape test, we observe that the reduced number of Newton iterations does not affect the discretization error estimate, nor the overall error estimate which is dominated by the discretization error estimate if the Newton algorithm is stopped.

\begin{figure}[t]
 \centering
\resizebox {\columnwidth} {!} {
 \begin{tikzpicture}
 \begin{axis}[
       xlabel = {$\vert\Th\vert = $ number of mesh elements }, 
       xmode=log,
       ymode=log,
      ]
   \addplot [RedOrange ,mark = square*] table [x = nbmare, y = eta_disc_resi] {tests/plate/HvM_comp.dat};
   \addplot [Sepia,mark = triangle*] table [x = nbmare, y = eta_ln_resi] {tests/plate/HvM_comp.dat};
   \addplot [blue,mark = diamond*] table [x = nbmari, y = eta_disc_rita] {tests/plate/HvM_comp.dat};
   \addplot [LimeGreen,mark = *] table [x = nbmari, y = eta_ln_rita] {tests/plate/HvM_comp.dat};
 \end{axis}
 \end{tikzpicture}
 \begin{tikzpicture}
 \begin{axis}[
       xlabel = {$\vert\Th\vert = $ number of mesh elements }, 
       legend entries = {{$\eta_{\rm{disc}}$,no},{$\eta_{\rm{lin}}$,no},{$\eta_{\rm{disc}}$, ad},{$\eta_{\rm{lin}}$, ad}},
       legend style={at={(0.03,0.53)},anchor=west},
       xmode=log,
       ymode=log,
      ]
   \addplot [RedOrange ,mark = square*] table [x = nbmare, y = eta_disc_resi] {tests/plate/Endo_comp.dat};
   \addplot [Sepia,mark = triangle*] table [x = nbmare, y = eta_ln_resi] {tests/plate/Endo_comp.dat};
   \addplot [blue,mark = diamond*] table [x = nbmari, y = eta_disc_rita] {tests/plate/Endo_comp.dat};
   \addplot [LimeGreen,mark = *] table [x = nbmari, y = eta_ln_rita] {tests/plate/Endo_comp.dat};
 \end{axis}
 \end{tikzpicture}
  \raisebox{2\height}{
\begin{tabular}{|c|c|c|c|c|}
 \hline
 & \multicolumn{2}{c|}{Hencky--Mises}  & \multicolumn{2}{c|}{Damage} \\
 \hline
 &  nor. &  adap. &  nor. & adap. \\
 \hline
0 & 5 & 3 & 6 & 4\\
1 & 5 & 4 & 6 & 4 \\
2 & 6 & 4 & 6 & 5 \\
3 & 6 & 5 & 7 & 5 \\
4 & 6 & 5 & 7 & 6 \\
5 & 7 & 5 & 8 & 6 \\
6 & 6 & 5 & 8 & 6 \\
 \hline
\end{tabular}}
}
\caption{Notched specimen plate. Comparison of the global discretization and linearization error estimators without and with adaptive stopping criterion for the Hencky--Mises model (left) and the damage model (middle), and comparison of the number of perfomed Newton iterations (right).}
\label{fig:plaque_Newton}
\end{figure}
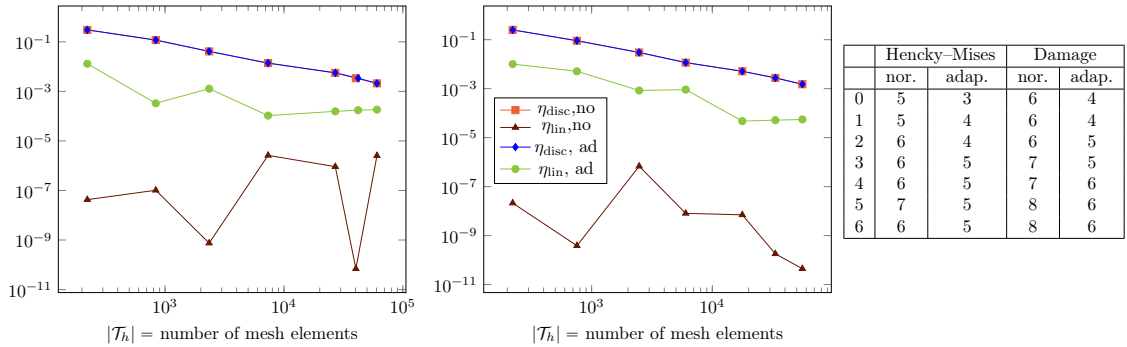


\section{Conclusions}
In this work we have developed an a posteriori error estimate for a wide class of hyperelastic problems. 
The estimate is based on stress tensor reconstructions and thus independent of the stress-strain relation, except for two constants. 
In a finite element software providing different mechanical behavior laws it can be directly applied to any of these laws.
The assumptions we make on the stress-strain relation are only used to obtain the equivalence of the energy norm and the dual norm of the residual of the weak formulation. 
Using the latter as error measure, the method can be applied to a wider range of behavior laws.
Exploring both numerical tests, we have promising results for general plasticity and damage models.
These results come at the price of solving local mixed finite element problems at each iteration of the linearization solver. 
In practice, the corresponding saddle point problems can be transformed into symmetric positive definite systems using the spaces of Section~\ref{sec:efficiency}. Furthermore, these matrices (or their decomposition) can be computed once in a preprocessing stage, and only need to be recomputed if one or more elements in the patch have changed due to remeshing.

\section*{Acknowledgements}
The authors thank Wietse Boon for interesting discussions about the Arnold--Falk--Winther spaces during the IHP quarter on Numerical Methods for PDEs in Paris. The authors also thank Kyrylo Kazymzrenko for providing his expertise on solid mechanics and for his help for designing the numerical test cases. The work of M. Botti was partially supported by Labex NUMEV (ANR-10-LABX-20) ref. 2014-2-006 and by the Bureau de Recherches G\'{e}ologiques et Mini\`{e}res.
\begin{footnotesize}
  \bibliographystyle{plain}
  \bibliography{neh_a_post}
\end{footnotesize}

\end{document}